\documentclass[10pt,letterpaper,oneside,english]{amsart}
\usepackage[utf8]{inputenc}
\usepackage{amstext}
\usepackage{amsthm}
\usepackage{esint}

\makeatletter

\pdfpageheight\paperheight
\pdfpagewidth\paperwidth

\numberwithin{equation}{section}
\numberwithin{figure}{section}
\theoremstyle{plain}
\newtheorem{thm}{\protect\theoremname}[section]
\theoremstyle{definition}
\newtheorem{defn}[thm]{\protect\definitionname}
\theoremstyle{remark}
\newtheorem{rem}[thm]{\protect\remarkname}
\theoremstyle{plain}
\newtheorem{prop}[thm]{\protect\propositionname}
\theoremstyle{plain}
\newtheorem{lem}[thm]{\protect\lemmaname}
\theoremstyle{plain}
\newtheorem{cor}[thm]{\protect\corollaryname}

\include{mypackages}
\include{mymacros}
 \usepackage{indentfirst}

\usepackage{amssymb}

\newcommand{\poincare}{Poincar\acute{e}}

\newcommand{\holder}{H\ddot{o}lder}

\makeatother

\usepackage{babel}
\providecommand{\corollaryname}{Corollary}
\providecommand{\definitionname}{Definition}
\providecommand{\lemmaname}{Lemma}
\providecommand{\propositionname}{Proposition}
\providecommand{\remarkname}{Remark}
\providecommand{\theoremname}{Theorem}

\begin{document}

\title{Constant $Q$-curvature metrics on conic 4-manifolds}

\author{Hao Fang}

\thanks{Part of Hao Fang's work is supported by a Simons Collaboration Grant.
Part of Biao Ma's work is supported by Graduate College Summer Fellowship
of University of Iowa.}

\address{14 MacLean Hall, University of Iowa, Iowa City, IA, 52242}

\email{hao-fang@uiowa.edu}

\author{Biao Ma}

\email{biaoma@uiowa.edu}
\begin{abstract}
We consider the constant $Q$-curvature metric problem in a given
conformal class on a conic 4-manifold and study related differential
equations. We define subcritical, critical, and supercritical conic
4-manifolds. Following \cite{troyanov1991prescribing} and \cite{chang1995extremal},
we prove the existence of constant $Q$-curvature metrics in the subcritical
case. For conic 4-spheres with two singular points, we prove the uniqueness
in critical cases and nonexistence in supercritical cases. We also
give the asymptotic expansion of the corresponding PDE near isolated
singularities.
\end{abstract}

\maketitle

\section{Introduction }

In this paper, we study Branson's $Q$-curvature on conic 4-manifolds.
First, we give a brief introduction to the study of $Q$-curvature
in conformal geometry, especially on 4-manifolds. Then, we discuss
conic 4-manifolds. Finally, we present our main results.

Let $(M^{4},g)$ be a compact 4 dimensional smooth Riemannian manifold.
Branson's $Q$-curvature \cite{branson1991explicit} is defined as
\[
Q_{g}=-\frac{1}{6}\Delta R-\frac{1}{2}|Ric|^{2}+\frac{1}{6}R^{2},
\]
where $R$ and $Ric$ are the scalar curvature and the Ricci curvature
tensor of $g$, respectively. Similar to the role of the Gaussian
curvature in the surface theory, $Q$-curvature is related to the
geometry of 4-manifolds by the following Gauss-Bonnet-Chern formula:
\begin{equation}
\int_{M}QdV_{g}=8\pi^{2}\chi(M)-\int_{M}\frac{1}{4}|W|^{2}dV_{g},\label{eq:Gauss Bonnet}
\end{equation}
where $W$ is the Weyl tensor. 

Let 
\[
k_{g}=\int_{M}QdV_{g}.
\]
Since the Weyl tensor is locally conformally invariant, (\ref{eq:Gauss Bonnet})
implies that $k_{g}$ is a global conformal invariant of $M$. Suppose
that $g_{w}=e^{2w}g$ is another metric in the same conformal class,
where $w\in C^{\infty}$. The corresponding $Q$-curvature for $g_{w}$
is given by 
\[
P_{g}w+Q_{g}=e^{4w}Q_{g_{w}}.
\]
Here $P_{g}$ is the Paneitz operator
\[
P_{g}w=\Delta_{g}^{2}w+\mathrm{div}\left(\frac{2}{3}Rg-2Ric\right)dw.
\]
An important property of the Paneitz operator is its conformal transformation
law. Namely, $P_{g_{w}}u=e^{-4w}P_{g}$ for $g_{w}=e^{2w}g$. In \cite{branson1991explicit}\cite{branson1992estimates,chang1995extremal},
the authors have extensively studied the Paneitz operator, $Q$-curvature
and their relation to zeta functional determinants of conformally
covariant operators. It is worth noting that the Paneitz operator
is also a special case of more generally defined GJMS operators, cf.
\cite{graham1992conformally}. 

A natural question in conformal geometry is finding constant $Q$-curvature
metrics in a given conformal class. From an analytical point of view,
the problem is equivalent to the following
\begin{equation}
P_{g}w+Q_{g}=c\cdot e^{4w},\label{eq:Constant Q equation}
\end{equation}
where $c$ is a constant. 

We state some important results in this field. Chang and Yang \cite{chang1995extremal}
have established the existence of constant $Q$-curvature metic if
the conformal metric class satisfies following conditions: \renewcommand{\labelenumi}{\alph{enumi})}  
\begin{enumerate}
\item $P_{g}$ is non-negative, 
\item ${\rm {Ker}}\ P_{g}=\{constants\}$, 
\item $k_{g}<16\pi^{2}$.
\end{enumerate}
The number $16\pi^{2}$ comes from a sharp Moser-Trudinger type inequality
due to Adams\cite{adams1988sharp}, which is critical in the arguments
of Chang-Yang\cite{chang1995extremal}. We should remark that this
existence result covers many cases. In particular, Gursky \cite{Gursky1999}has
shown that conditions (a), (b), and (c) are satisfied when (i) Yamabe
constant $Y_{g}\geq0$, (ii) $k_{g}\geq0$, and (iii) $M$ not conformal
to $S^{4}$. If $k_{g}>16\pi^{2}$, Djadli and Malchiodi \cite{djadli2008existence}
are able to establish the existence of constant $Q$-curvature metrics
if $k_{g}\not=16l\pi^{2},l=1,2,3\cdots$ , by a delicate min-max argument.
For $k_{g}=16\pi^{2}$ , some existence results are obtained by J.
Li, Y. Li and P. Liu\cite{li2012q} under certain additional conditions.
In even dimension $n>4$, constant $Q$-curvature problems have been
studied in \cite{NDIAYE20071,BAIRD2009221,brendle2003global}.

We now introduce conic 4-manifolds. Let $(M^{4},g_{0})$ be a compact
smooth Riemannian 4-manifold. The conical singularities are represented
by the conformal divisor 
\[
D=\sum_{i=1}^{k}p_{i}\beta_{i},
\]
where $p_{i}\in M$ and $0>\beta_{i}>-1$. For simplicity, we assume
that $\beta_{1}\leq\beta_{2}\leq\cdots\leq\beta_{k}$. Let $\gamma(x)$
be a function in $C^{\infty}(M^{4}-\{p_{i}\})$ with following local
forms: in a neighborhood of $p_{i}$:
\begin{equation}
\gamma(x)=\sum_{i=1}^{k}\beta_{i}\log(r_{i})\eta_{i}(x)+f_{i}(x),\label{eq:singular term}
\end{equation}
where $r_{i}=\mathrm{dist}_{g_{0}}(x,p_{i})$ is the distance to $p_{i}$,
$f_{i}(x)$ is locally smooth near $p_{i}$, and $\eta_{i}(x)$ is
a local cut-off function near $p_{i}.$ Consider a singular metric
\begin{equation}
g_{D}=e^{2\gamma}g_{0}.\label{eq:g_D definition}
\end{equation}
$g_{D}$ is a metric conformal to $g_{0}$ on $M^{4}-\{p_{i}\}$ which
has a conical singularity at each $p_{i}$. Let $dV_{D}$ be the volume
element of $g_{D}$. Let $H^{2}(dV_{D})=W^{2,2}(dV_{D})$ be the corresponding
Sobolev space with respect to the measure $dV_{D}$. We define the
conformal class of $g_{D}$
\[
[g_{D}]:=\{g_{w}=e^{2w}g_{D}:w\in H^{2}(dV_{D})\cap C^{\infty}(M-\{p_{i}\})\}.
\]
Note $[g_{D}]$ depends only on $(M^{4},g_{0})$ and $D$. Function
$\log|x-p_{i}|$ is not in $H_{loc}^{2}(dV_{D})$ and neither is $\gamma(x)$.
Let $g_{1}\in[g_{D}]$. We call a 4-tuple $(M^{4},g_{0},D,g_{1})$
a \textbf{conic 4-manifold. }Here\textbf{ $g_{0}$ }is called the
background metric and $g_{1}$ is called the conic\textbf{ }metric.
We remark that our definition of conic manifolds can be generalized
to general dimensions. We also note here that our definition of conic
singularity is in the sense of Reimannian geometry. See, for example,
\cite{cheeger1996lower,cheeger1997structure}. In particular, our
definition is different from similar notation that is used in studies
of K$\ddot{a}$hler geometry.

We study the conformal geometry of conic manifolds. For conic 4-manifolds,
a corresponding Gauss-Bonnet-Chern formula first proved in \cite{buzano2015chern}
can be stated as follows:
\begin{equation}
k_{g_{1}}=\int_{M}Q_{g_{1}}dV_{g_{1}}=\int_{M}Q_{g_{0}}dV_{g_{0}}+8\pi^{2}\sum_{i=1}^{k}\beta_{i}.\label{eq:add1}
\end{equation}
We will also give a detailed proof directly following the line of
Troyanov\cite{troyanov1991prescribing} in Section 3.

In this paper, we study the constant $Q$-curvature metric for conic
4-manifolds. Our motivations are both analytical and geometrical.
Analytically, Moser-Trudinger-Adams inequality gives the key estimate
in \cite{branson1992estimates,chang1995extremal,djadli2008existence}.
We find out that that Moser-Trudinger-Adams inequality can be generalized
under our new setting. Geometrically, the study of stability condition
and conic singularities plays a central role in recent developments
in K$\ddot{a}$hler geometry\cite{Chen2014K,Chen2015K,Chen2013K,Tian1996,tian2015k}.
It is our intention to find a suitable ``stability'' condition in
conformal geometry. 

Finally, we describe our new results. For conic 4-manifold $(M,g_{0},D,g_{1})$,
we study the existence of a metric $g_{w}\in[g_{D}]$ such that the
$Q$-curvature of $g_{w}$ is a constant, which is equivalent to finding
(weak) solutions in $H^{2}(dV_{D})=W^{2,2}(dV_{D})$ of the equation:
\begin{equation}
P_{g_{D}}w+Q_{D}=c\cdot\exp(4w).\label{eq:The solution}
\end{equation}
In fact, $H^{2}(dV_{D})$ is equivalent to $H^{2}(dV_{0})$ space
of $g_{0}$ and by $\poincare$ inequality, the $H^{2}(dV_{0})$ norm
of $w$ is given by 
\[
\|u\|_{H^{2}(dV_{0})}^{2}=\int_{M}uP_{g_{0}}udV_{0}+\|u\|_{L^{2}(dV_{0})}^{2},
\]
when conditions (a) and (b) are satisfied, see Section 2.

Our approach to study conic 4-manifolds comes from pioneering works
on conic surfaces by Troyanov and others. Historically, Troyanov \cite{Troyanov1989}\cite{troyanov1991prescribing}
systematically studied the prescribed curvature problems on conic
surfaces. In particular, he \cite{troyanov1991prescribing} classified
conformal metrics on conic Riemann surfaces into three categories:
subcritical, critical, and supercritical. He showed that in subcritical
cases, there is a unique constant Gaussian curvature metric. By a
geometric construction, Luo and Tian \cite{LuoTian} proved that with
more than 2 conic points the solution exists if and only if in subcritical
case. Chen and Li \cite{chen1995kinds}proved the same results by
classifying solutions of the corresponding PDE. Chen and Li also proved
that in the critical case, the only solutions are radial symmetric,
like an American football. In a recent work, the first named author
and Lai described the limiting process when a subcritical metric deforms
continuously towards a critical one\cite{Fang2016}. 

Motivated by Troyanov's classification, we give the following definition.
\begin{defn}
\renewcommand{\labelenumi}{\roman{enumi})}  Let $D=\sum_{i=1}^{k}\beta_{i}p_{i}$
and $\beta_{1}=\min\{\beta_{i}\}$. A conic 4-manifold $(M,g_{0},D,g_{1})$ 
\end{defn}

\begin{enumerate}
\item is called \textbf{subcritical, }if\textbf{ $k_{g_{0}}+8\pi^{2}(\sum_{i}\beta_{i})<8\pi^{2}(2+2\beta_{1})$; }
\item is called \textbf{critical,} if\textbf{ $k_{g_{0}}+8\pi^{2}(\sum_{i}\beta_{i})=8\pi^{2}(2+2\beta_{1})$;}
\item is called \textbf{supercritical, }if\textbf{ $k_{g_{0}}+8\pi^{2}(\sum_{i}\beta_{i})>8\pi^{2}(2+2\beta_{1})$.}
\end{enumerate}
\begin{rem}
For general even dimension $n$, similarly, we can define conic $n$-manifold
$(M^{n},g_{0},D,g_{1})$ with $D=\sum_{i=1}^{k}\beta_{i}p_{i}$. Let
$\beta_{1}=\min\{\beta_{i}\}$. Let $Q_{g_{0}}$ be the corresponding
$Q$-curvature. Set $k_{g_{0}}=\int_{M}Q_{g_{0}}dV_{g}$. Then, one
can define $M$ to be a subcritical conic manifold if $k_{g_{0}}+\gamma_{n}(\sum_{i=1}^{k}\beta_{i})<\gamma_{n}(2+2\beta_{1})$,
a critical conic manifold if $k_{g_{0}}+\gamma_{n}(\sum_{i=1}^{k}\beta_{i})=\gamma_{n}(2+2\beta_{1})$,
or a supercritical conic manifold if $k_{g_{0}}+\gamma_{n}(\sum_{i=1}^{k}\beta_{i})>\gamma_{n}(2+2\beta_{1})$,
where $\gamma_{n}=\frac{(n-1)!|S_{n}|}{2}$.\label{rem:For-general-even-n-subcritical-case}
\end{rem}

In this paper, we primarily consider subcritical cases for general
4-manifolds and critical cases on the sphere $S^{4}$. 

Our first result is the following:
\begin{thm}
\label{thm:Subcritical}Let $(M^{4},g_{0},D,g_{D})$ be a conic 4-manifold.
Suppose that $P_{g_{0}}$ is nonnegative and its kernel contains only
constant functions on $(M^{4},g_{0}$). Let $\beta_{1}=min\{\beta_{i}\}$.
Suppose $k_{g_{0}}+8\pi^{2}(\sum_{i}\beta_{i})<8\pi^{2}(2+2\beta_{1})$.
Then there is a conformal metric $g_{w}\in[g_{D}]$ represented by
$w\in C^{\tau}(M)\cap C^{\infty}(M-\{p_{i}\})$ for $0<\tau<min\{2,4(1+\beta_{1})\}$
such that $(M^{4},g_{w})$ has constant $Q$-curvature.
\end{thm}

\begin{rem}
After the draft of this paper, it was brought to our attention that
Hyder-Lin-Wei \cite{hyder20193} have studied the solutions of the
equation (1.9) on spheres. See also Maalaoui \cite{maalaoui2016prescribing}
for a different approach. Hyder-Lin-Wei \cite{hyder20193} give a
solvability condition which is equivalent to our subcritical condition
on $S^{4}$. Theorem 1.3 also covers non-flat cases. Hence, our results
are more general geometrically. Interestingly, we achieve the same
solvability condition for standard 4-sphere via different approaches.
Inspired by \cite{hyder20193}, we can apply the same techniques of
this paper to prove the existence of constant $Q$-curvature metrics
in the subcritical cases defined in Remark \ref{rem:For-general-even-n-subcritical-case}
with two assumptions a) Paneitz operator $P_{g}\geq0$ and b) ${\rm Ker}P_{g}=\{constants\}$.
However, it seems difficult to check a) and b) except for conformally
flat manifolds.
\end{rem}

We summarize our approach to prove Theorem \ref{thm:Subcritical}.
In the subcritical case, following \cite{chang1995extremal} and \cite{troyanov1991prescribing},
we study the ${\rm II}$ functional defined in \cite{branson1991explicit}.
For the ${\rm II}$ functional of a conic 4-manifold, we obtain the
coerciveness and hence the existence of a minimizer. We use a conic
version of Moser-Trudinger-Adams inequality, which on Euclidean domains
has been proved by Lam and Lu \cite{Lam-Lu}. Following \cite{branson1992estimates}
and \cite{chang1995extremal}, we generalize this inequality to general
conic manifolds.

Following works of Troyanov, we also consider solutions of (\ref{eq:The solution})
with 2 singular points on conic 4 spheres. By a stereographic projection
and a conformal transformation, we consider the PDE on the punctured
Euclidean space $\mathbb{R}^{4}\backslash\{0\}$, 
\begin{equation}
\Delta^{2}u=e^{4u},\label{eq: PDE in R^4}
\end{equation}
where $u$ has singularities at 0 and $\infty$. By our definition
of conic singularities, $u$ behaves logarithmically at $0$ and infinity.
Note $-\frac{1}{8\pi^{2}}\log|x-p|$ is the Green's function for bilaplacian
on $\mathbb{R}^{4}$ at $p\in\mathbb{R}^{4}$. In order to emphasize
singularities in (\ref{eq: PDE in R^4}), we can equivalently write
the following PDE on $\mathbb{R}^{4}:$
\begin{align}
\Delta^{2}u & =e^{4u}+8\pi^{2}\beta_{0}\delta_{0},\label{eq:2 singualrities, R^4 equation with 2 singular source term}
\end{align}
where $\delta_{0}$ is the Dirac measure at $0$. We further require
that as $|x|\to\infty$, 
\[
u(x)\sim-(2+\beta_{1})\log|x|,
\]
to indicate the conic singularity at infinity. (\ref{eq:2 singualrities, R^4 equation with 2 singular source term})
is a special form of the following
\begin{align}
\Delta^{2}u(x) & =e^{4u(x)}+8\pi^{2}\sum_{i=1}^{k-1}\beta_{i}\delta_{p_{i}}(x),\label{eq:R^4 equation with singular source term}\\
u(x) & \sim-(2+\beta_{k})\log|x|,\ |x|\to\infty,\nonumber 
\end{align}
where $\delta_{p_{i}}$ is the Dirac measure at $p_{i}$. Each solution
of (\ref{eq:R^4 equation with singular source term}) represents a
conic constant $Q$-curvature metric on $S^{4}$ with conformal divisor
$D=\sum_{i=1}^{k-1}p_{i}\beta_{i}+\beta_{k}\infty,$ where $\infty$
is the north pole on $S^{4}$, due to the stereographic projection.

To study (\ref{eq: PDE in R^4}), we first discuss radial symmetric
solutions. Using a cylindrical coordinate, we reduce the PDE to a
4$^{th}$ order ODE. 
\begin{equation}
v''''(t)-4v''(t)=e^{4u},\label{eq:ODE}
\end{equation}
where $t\in\mathbb{R}.$ We have the following:
\begin{thm}
\label{thm:ODE}There is a family of solutions $v_{\alpha}$ of (\ref{eq:ODE}),
parametrized by $\alpha=1+\beta>0$ such that $v'_{\alpha}(t)$ goes
to $\pm\alpha$ as $t$ goes to $\pm\infty$\label{Them:ODE result}.
Differing by a constant and a translation in $t$, these are the only
solutions with linear growth at infinity.
\end{thm}

We remark that in a paper \cite{Hyder2019} by Hyder, Mancini and
Martinazzi, the existence of radial symmetric solutions for constant
Q-curvature metric for $S^{n}$, $n\geq4$ is established, which includes
Theorem \ref{thm:ODE} as a special case. We thank a referee for pointing
this out. Our proof is independent and different in flavor.

Furthermore, we have the following uniqueness theorem:
\begin{thm}
\label{thm:Uniquness result}A constant Q-curvature 4-sphere with
2-singular points must be a radial symmetric conic sphere. Both singularities
have the same index. Under a cylindical coordinate, the metric is
given by one of the solutions described in Theorem \ref{Them:ODE result}
up to a translation.
\end{thm}

In the smooth case, the radial symmetry of solutions of (\ref{eq: PDE in R^4})
has been established by Lin \cite{lin1998classification} using the
moving plane method. Lin investigates the asymptotic behavior of the
solution $u$ and proves that $\Delta u$ has an asymptotic harmonic
expansion at infinity. This expansion implies certain monotonicity
of $u$ and allows one to initiate moving plane method near infinity.
We should mention that Caffarelli, Gidas, and Spruck \cite{caffarelli1989asymptotic}
were first to investigate such expansions in order to apply moving
plane method to study the semi-linear elliptic equation $-\Delta u=u^{\frac{n+2}{n-2}}$.
Since we allow singularities at the origin and at the infinity, we
do not expect to have an exact expansion like those in Lin \cite{lin1998classification}.
However, by careful analysis, we establish a similar asymptotic expansion
near each singularity. It generalizes our regularity theorem in Theorem
\ref{thm:Subcritical}. See details in Section 7.

We intend to investigate critical cases and supercritical cases further.
We expect some nonexistence results for constant $Q$-curvature metric
in supercritical cases and critical cases with more than 3 singular
points. General supercritical cases are more elusive since we have
the existence result \cite{djadli2008existence}. 

In a recent work \cite{fang2019sigma}, the first named author and
Wei derive a similar criterion for the existence of constant $\sigma_{2}$
curvature metrics on conic 4-manifolds. In particular, they establish
the nonexistence result for supercritical cases and uniqueness result
for critical cases. Combined with \cite{fang2019sigma}, our results
imply rich conformal geometry of conic manifolds and indicate an interesting
direction. In the future, we would like to explore corresponding topics
for more general types of singular manifolds.

We organize the paper as follows. In Section 2, we discuss some function
spaces and embedding theorems with conical singularities. In Section
3, we derive a Gauss-Bonnet-Chern formula resembling the one for Riemann
surfaces. In Section 4, a Moser-Trudinger-Adams type inequality for
conic manifolds is established. In Section 5, we give the proof of
Theorem \ref{thm:Subcritical}. In Section 6, we study the radial
symmetric solutions and prove Theorem \ref{Them:ODE result}. In Section
7, a detailed asymptotic expansion of solutions near a conical singularity
is given. In Section 8, we use the asymptotic expansion from Section
7 to prove Theorem \ref{thm:Uniquness result}.

We would like to thank referees for pointing out some previous works
\cite{buzano2015chern} and \cite{Hyder2019}, which were unknown
to us. We would like to thank Alice Chang and Paul Yang for their
interest in this work. The second named author would like to thank
Mijia Lai for help and comments. Part of this work was done during
the second named author visiting Shanghai Jiaotong University. He
would like to thank the hospitality. 

\section{Function spaces}

In this section, we discuss several function spaces in our study and
list their relations.

Let $(M^{4},g_{0},D,g_{D})$ be a $4$ dimensional conic manifold
according to the definitions similar to (\ref{eq:singular term})
and (\ref{eq:g_D definition}) where $D=\sum_{i=1}^{k}\beta_{i}p_{i}$.
Likewise, we assume that 
\[
g_{D}=e^{2\gamma(x)}g_{0}.
\]
$\gamma$ is given by:
\begin{equation}
\gamma(x)=\sum\beta_{i}\log(r_{i})\eta_{i}(x)+f_{i}(x),\label{eq:Singular weight gamma; functional spaces}
\end{equation}
where $f_{i}(x)$ is smooth and $\eta_{i}(x)$ is a smooth cutoff
function near $p_{i}$. Let $g_{1}=g_{D}$ and $dV_{i}$ be the volume
elements of $g_{i}$ for $i=0,1$ respectively. We can define $H^{2}(dV_{i})$
norm of a function $u$ in $C^{\infty}(M)$ by
\[
\|u\|_{H^{2}(dV_{i})}^{2}=\int_{M}|u|^{2}dV_{i}+\int_{M}|\nabla_{g_{i}}u|^{2}dV_{i}+\int_{M}|\nabla_{g_{i}}^{2}u|^{2}dV_{i},
\]
for $i=0,1$. Let $H^{2}(dV_{i})$ be the closure of $C^{\infty}(M)$
under the $H^{2}(dV_{i})$ norm. Although $g_{1}$ is not smooth,
the related $L^{p}$ spaces have certain comparison theorem and $H^{2}(dV_{i})$
are actually the same. This is in fact a crucial point in Troyanov's
work\cite{troyanov1991prescribing}. Besides, the $\poincare$ inequality
and compact embedding theorem for Sobolev space are still valid. The
following results are similar to those given in \cite{troyanov1991prescribing}
which concern surfaces.
\begin{prop}
(Weighted Sobolev inequality) \label{thm:(Weighted-Sobolev-inequality)}Let
$\Omega$ be an open domain in $\mathbb{R}^{4}$ contains $0$. Let
$\beta>-1$. Then there is some constant $C(\Omega)$, independent
of $p$ such that for any $u\in C_{c}^{2}(\Omega)$, 
\[
\left(\int_{\Omega}|u|^{p}\cdot|x|^{4\beta}dx\right)^{\frac{1}{p}}\leq C(\Omega)p^{\frac{1}{2}}\cdot||\Delta u||_{L^{2}(\Omega)}
\]
\end{prop}

\begin{proof}
See the Appendix in \cite{troyanov1991prescribing}. 
\end{proof}
By partition of unity and Proposition \ref{thm:(Weighted-Sobolev-inequality)},
we have Sobolev's embedding from $H^{2}(dV_{i})$ spaces to $L^{p}(dV_{i})$
spaces as a natural extension.

\begin{prop}
(\label{prop:(Sobolev's-embedding)-There}Sobolev's embedding) There
is a constant $C$ such that for all $u\in H^{2}(dV_{i})$ and $p\in[1,\infty)$,
we have $\|u\|_{L^{p}(dV_{i})}\leq C\sqrt{p}\|u\|_{H^{2}(dV_{i})}$.
\end{prop}

\begin{prop}
\label{thm:(-comparison).-If}($L^{p}$ comparison). Let $\alpha=\min\{\beta_{i}+1\}$
and $\omega=\max\{\beta_{i}+1\}$. If $p>\frac{q}{\alpha}$, then
$L^{p}(dV_{0})\subset L^{q}(dV_{1}).$ If $p>q\omega$, then $L^{p}(dV_{1})\subset L^{q}(dV_{0})$.
\end{prop}

\begin{proof}
See \cite{troyanov1991prescribing}. 
\end{proof}
\begin{prop}
\label{prop:.invariance of H2}$H^{2}(dV_{0})=H^{2}(dV_{1})$. 
\end{prop}

\begin{proof}
By Poincaré's inequality, $H^{2}(dV_{i})$ norm is given by $\|u\|_{L^{2}(dV_{i})}+\|\Delta u\|_{L^{2}(dV_{i})}$.
We claim that $\|\Delta_{g_{0}}u\|_{L^{2}(dV_{0})}$ and $\|\Delta_{g_{1}}u\|_{L^{2}(dV_{1})}$
are equivalent. In fact, we have that 
\[
e^{2\gamma(x)}\Delta_{g_{1}}u(x)=\Delta_{g_{0}}u(x)+2\nabla_{g_{0}}u(x)\cdot\nabla_{g_{0}}\gamma(x).
\]
Observe that, $|\nabla_{g_{0}}\gamma(x)|\sim\beta_{i}|x-p_{i}|^{-1}$
at a neighborhood of $p_{i}$ and smooth elsewhere. By the well known
Hardy's inequality, 
\[
\|\nabla_{g_{0}}u(x)\cdot\nabla_{g_{0}}\gamma(x)\|_{L^{2}(dV_{0})}\leq C\|\nabla_{g_{0}}^{2}u\|_{L^{2}(dV_{0})}.
\]
Hence, $\|\Delta_{g_{1}}u\|_{L^{2}(dV_{1})}<C\|\Delta_{g_{0}}u\|_{L^{2}(dV_{1})}$.
The other direction is the same. 

We only have to show
\begin{equation}
\|u\|_{L^{2}(dV_{0})}\leq C\|u\|_{H^{2}(dV_{1})}\ and\ \|u\|_{L^{2}(dV_{1})}\leq C\|u\|_{H^{2}(dV_{0})}.\label{eq: H^2 equivalence}
\end{equation}
Let $\alpha=\min\{\beta_{i}+1\}$ and $\omega=\max\{\beta_{i}+1\}$.
By Proposition \ref{thm:(-comparison).-If}, we can choose $p>2\omega$
to get $\|u\|_{L^{2}(dV_{0})}\leq\|u\|_{L^{p}(dV_{1})}$. Then we
use Proposition \ref{prop:(Sobolev's-embedding)-There} to get $\|u\|_{L^{2}(dV_{0})}\leq C\|u\|_{H^{2}(dV_{1})}$.
The second inequality in (\ref{eq: H^2 equivalence}) can be proved
by choosing $p>\frac{2}{\alpha}$.
\end{proof}
Since we have established Proposition 2.4, from now on, we do not
distinguish $H^{2}(dV_{i})$ . 

The following two propositions are quite standard. See \cite{Troyanov1989}
for details.

\begin{prop}
(compact embedding) \label{prop:(compact-embedding)-The}The embedding
$H^{2}(dV_{i})\hookrightarrow L^{p}(dV_{i})$ for $i=0,1$ is compact
for $1<p<\infty$.
\end{prop}

\begin{prop}
( \label{prop:(Poincar=0000E9's-inequality)-If} Poincaré's inequality)
If $\int_{M}vdV_{i}=0$, $i=0,1$ , then $\|v\|_{H^{2}}\leq C\|\Delta v\|_{2}$
for some constant $C=C(M)$.
\end{prop}

\section{Gauss-Bonnet-Chern formula}

In this section, we establish the conformal Gauss-Bonnet-Chern formula
(\ref{eq:add1}), which is originally proved in \cite{buzano2015chern}.
Our discussion follows the approach of Troyanov\cite{troyanov1991prescribing}
and notations in previous sections.

On $(M^{4},g)$, recall that the Gauss-Bonnet-Chern formula 
\[
\int_{M}Q(x)dV_{g}(x)+\frac{1}{4}\int_{M}|W(x)|^{2}dV_{g}(x)=8\pi^{2}\chi(M).
\]
In this section, we will describe the contribution of conical singularities
to this Gauss-Bonnet-Chern formula.

Let $(M^{4},g_{0},D,g_{D})$ be a conic 4-manifold and $g_{1}=g_{D}$.
A good choice of base metric $g_{0}$ will simplify the discussion.
Thus we use the conformal normal coordinates by \cite{lee1987yamabe}.
We can find a metric $g\in[g_{0}]$ such that around each given point
$p_{i}$, the $normal$ coordinates of $g$ satisfy 
\begin{equation}
det(g(x))=1+O(|x|^{N})\label{eq:conformalnormalcoordinates}
\end{equation}
for any $N\in\mathbb{N}$. 
\begin{lem}
Suppose that $B_{\epsilon}(x)$ is a ball with radius $\epsilon$
centered at $x\in M$. Let $\delta$ be the injective radius of $x\in M$
Let $h(x,y)=\log(|x-y|)f(y)$ where $f(y)$ is a smooth function supported
in $B_{\delta}(x)$ and equals $1$ in $B_{\epsilon}(x)$. Then \label{lem:Log}
\[
\lim_{\epsilon\to0}\int_{M\backslash B_{\epsilon}(x)}P_{y}(h(x,y))dV_{g}(y)=8\pi^{2}=4|S^{3}|,
\]
where $P_{y}$ is the Paneitz operator with respect to $y$ and $|S^{3}|$
is the volume of a unit 3-sphere.
\end{lem}

\begin{proof}
Let $(r,\theta^{i})$ be the normal coordinates at $x.$ Let $f\in C^{2}(M-\{x\})$.
If $f=f(r)$ then
\[
\Delta f(r)=f''+\frac{3}{r}f'+f'\partial_{r}\log\sqrt{det(g)},
\]
where $f'$ denotes the derivative with respect to $r$. In particular,
$\Delta\log r=\frac{2}{r^{2}}+\frac{1}{2r}\partial_{r}\log det(g)$.
Suppose that $Pu=\Delta^{2}u+\mathrm{div}(A_{g}du)$ where $A_{g}=\frac{2}{3}Rg-2Ric$.
Divergence theorem then gives
\begin{align}
\int_{M\backslash B_{\epsilon}(x)}P_{y}(h(x,y))dV_{g}(y) & =-\int_{\partial B_{\epsilon}(x)}\left(\frac{\partial}{\partial r}\Delta\log r+O(r^{-2})\right)d\Omega_{g}\nonumber \\
 & =\int_{\partial B_{\epsilon}}\frac{4}{\epsilon^{3}}d\Omega_{g}+o(1),\label{eq:Gauss bonnet Lemma 1}
\end{align}
since $\partial_{r}\Delta\log r=-\frac{4}{r^{3}}+O(r^{N-3})$ by our
assumption (\ref{eq:conformalnormalcoordinates}).  Take $\epsilon\to0$
then the right hand side of (\ref{eq:Gauss bonnet Lemma 1}) approaches
$8\pi^{2}$ since the unit sphere $S^{3}$ has volume $2\pi^{2}$.
\end{proof}
\begin{prop}
(Gauss-Bonnet-Chern) \label{prop:(Gauss-Bonnet-Chern)-Suppose-tha}Suppose
that $g_{1}=e^{2\gamma}g_{0}$ is the metric with $k$ singular points
given by divisor $D=\sum_{i=1}^{k}p_{i}\beta_{i}$ , where $p_{i}\in M$,
$\beta_{i}>-1$. Suppose that $\gamma(x)=\beta_{i}\eta_{i}\log r+f(x)$
in a neighborhood of $p_{i}$, $r=dist(p_{i},x)$ and $f(x)\in H^{2}(M)\cap C^{\infty}(M-\{p_{i}\})$.
Then we have the formula
\[
\int_{M}Q_{g_{1}}dV_{g_{1}}=\int_{M}Q_{g_{0}}dV_{g_{0}}+8\pi^{2}\left(\sum_{i=1}^{k}\beta_{i}\right)
\]
\end{prop}

\begin{proof}
First, we assume that $f$ is $C^{\infty}(M)$. Observe that 
\begin{align*}
\int_{M}Q_{g_{1}}dV_{g_{1}} & =\int_{M}e^{-4\gamma}(P_{g_{0}}\gamma+Q_{g_{0}})dV_{g_{1}}\\
 & =\int_{M}\left(P_{g_{0}}\gamma+Q_{g_{0}}\right)dV_{g_{0}}.
\end{align*}
We need to compute $\int_{M}P_{g}\gamma dV_{g}$. Suppose that $\epsilon$
is a positive real number smaller than the injective radius at each
$p_{i}$. Let $B_{i}=B_{\epsilon}(p_{i})$ be the ball with radius
$\epsilon$ at $p_{i}$. We consider the integral $\int_{M-\cup B_{i}}P_{g_{0}}\gamma dV_{g_{0}}$.
By divergence theorem, we see that this integral is just 
\[
\sum_{i=1}^{k}\int_{\partial B_{i}}\left(\frac{\partial}{\partial n}(\Delta_{g_{0}}\gamma)+\langle n,A_{g_{0}}(\nabla\gamma)\rangle\right)d\Omega_{g_{0}},
\]
where $d\Omega_{g_{0}}$ is the area element on the geodesic sphere.
Note $\gamma(x)=f(x)+\beta_{i}\eta_{i}\log(r)$ near $p_{i}$. Thus,
by Lemma \ref{lem:Log}, we see that 
\[
\lim_{\epsilon\to0}\int_{M-\cup B_{\epsilon}(p_{i})}P_{g_{0}}\gamma dV_{g_{0}}=\sum_{i=1}^{k}8\pi^{2}\beta_{i}.
\]
For more general $f\in H^{2}(M)\cap C^{\infty}(M-\{p_{i}\})$, let
$\phi_{\epsilon}$ be a smooth function such that $supp(1-\phi_{\epsilon})\subset\cup_{i=1}^{k}B_{2\epsilon}(p_{i})$
and $\phi_{\epsilon}=0$ on $B_{\epsilon}(p_{i})$. We can assume
that $|D^{k}\phi_{\epsilon}|<C\epsilon^{-k}$ for $k=1,2,3,4$. Now
it suffices to prove
\begin{equation}
\lim_{\epsilon\to0}\int_{M}\phi_{\epsilon}P_{g_{0}}fdV_{g_{0}}=0.\label{eq:Gauss bonnet holds for general situation}
\end{equation}
 Let $B_{\epsilon}=\cup_{i}B_{\epsilon}(p_{i})$. Note that the highest
order term of $P_{g_{0}}$ can be estimated by $\holder's$ inequality
\begin{align}
|\int_{M}\Delta f\Delta\phi_{\epsilon}dV_{g_{0}}| & \leq|\int_{B_{2\epsilon}}C\epsilon^{-2}\Delta fdV_{g_{0}}|\nonumber \\
 & \leq C|\int_{B_{2\epsilon}}\epsilon^{-4}dV_{g_{0}}|^{\frac{1}{2}}\left|\int_{B_{\epsilon}}(\Delta f)^{2}dV_{g_{0}}\right|^{\frac{1}{2}}\nonumber \\
 & \leq C\|f\|_{H^{2}(B_{\epsilon})}\to0\quad as\ \epsilon\to0.\label{eq: general gauss bonnet higher order}
\end{align}
The lower order terms can be estimated by 
\begin{align}
\int_{M}\phi_{\epsilon}\mathrm{div}(A_{g_{0}}df)dV_{g_{0}} & =\int_{M}A_{g_{0}}(\nabla\phi_{\epsilon},\nabla f)dV_{g_{0}}=\int_{B_{2\epsilon}}A_{g_{0}}(\nabla\phi_{\epsilon},\nabla f)dV_{g_{0}}\label{generela gauss bennet lower oder}\\
 & \leq\int_{B_{2\epsilon}}||A_{g_{0}}||_{\infty}\cdot|\nabla\phi_{\epsilon}|\cdot|\nabla f|dV_{g_{0}}\nonumber 
\end{align}
By $\holder's$ inequality, the lower order terms go to $0$ as $\epsilon$
goes to $0$. (\ref{eq: general gauss bonnet higher order}) and (\ref{generela gauss bennet lower oder})
proves (\ref{eq:Gauss bonnet holds for general situation}) which
concludes the whole proof.
\end{proof}
For later use, we state the following corollary.
\begin{cor}
Suppose that in a neighborhood of $p_{i}$, $\gamma(x)=\beta_{i}\eta_{i}\log(r)+f(x)$
for some smooth $f(x)$. Then $Q_{1}e^{4\gamma}\in L^{p}(dV_{0})$
for $1<p<2$ and $Q_{1}\in L^{2}(dV_{1})$ \label{cor:-for-.lemforregularity Q_rho}.
\end{cor}

\begin{proof}
Recall $P_{g}v=\Delta_{g}^{2}v+\mathrm{div}(A_{g}dv)$. As in the
definition (\ref{eq:singular term}) in a neighborhood of $p_{i}$,
\[
\gamma(x)=\beta_{i}\log(r)+f(x),
\]
 for some smooth $f(x)$. By calculation in Lemma \ref{lem:Log},
we see that, 
\begin{equation}
\Delta_{g_{0}}^{2}\gamma(x)=O(r^{N-4}).\label{eq:Leading term in gamma; cor3.3}
\end{equation}
Recall the leading term of $P_{g_{0}}$ is $\Delta_{g_{0}}^{2}$.
Therefore, we obtain from (\ref{eq:Leading term in gamma; cor3.3})
that 
\begin{equation}
|P_{g_{0}}\gamma|\leq Cr^{-2}.\label{eq:Pg_0 estimate cor 3.3}
\end{equation}
Thus, we conclude from (\ref{eq:Leading term in gamma; cor3.3}) that
\begin{equation}
Q_{1}e^{4\gamma}=(P_{g_{0}}\gamma+Q_{0})\sim O(r^{-2}),\label{eq:Q_1=00005Crho;cor 3.3}
\end{equation}
which implies that $Q_{1}e^{4\gamma}\in L^{p}(dV_{0})$ for $1<p<2$.
Moreover, we write (\ref{eq:Q_1=00005Crho;cor 3.3}) as $Q_{1}\sim O(r^{-2-4\beta_{i}})$,
then 
\[
Q_{1}^{2}e^{4\gamma}\sim O(r^{-4-4\beta_{i}}).
\]
Since $-4-4\beta_{i}>-4$, we have $Q_{1}\in L^{2}(dV_{1})$. 
\end{proof}

\section{A Modified Moser-Trudinger-Adams Inequality}

In this Section, we establish a conic Moser-Trudinger-Adams inequality.

A singular version of Moser-Trudinger-Adams inequality on bounded
domains in Euclidean spaces has been proved by Lam and Lu \cite{Lam-Lu}.
In this section, we first give a quick proof based on a comparison
principle of Talenti\cite{talenti1976elliptic} and a lemma by Tarsi
\cite{tarsi2012adams}. Then we establish a corresponding inequality
for general conic 4-manifolds.

We first introduce Talenti's comparison principle. Let $f$ be a measurable
function with support in a bounded domain $\Omega\subset\mathbb{R}^{n}$.
Let $\lambda(s)=m(\{x:|f(x)|>s\})$ and $f^{*}(t)=sup\{s>0:\lambda(s)>t\}$.
The spherical rearrangement $f^{\#}(x)$ of $f$ is defined to be
\[
u^{\#}(x)=u^{*}(\omega_{n}|x|^{n}),\ x\in\Omega^{\#}.
\]
Here $\Omega^{\#}$ is a open ball in $\mathbb{R}^{n}$ with the same
measure as $\Omega$, $\omega_{n}$ is the volume of a unit $n$ dimensional
ball. Let
\begin{equation}
b_{n,2}=\frac{1}{\omega_{n}}\left[\frac{4\pi^{\frac{n}{2}}}{\Gamma(\frac{n}{2}-1)}\right]^{\frac{n}{n-2}},\label{eq:definition of b_n,2}
\end{equation}
following \cite{adams1988sharp}. Note that $\omega_{n}=\pi^{n/2}/\Gamma(\frac{n}{2}+1)$.
So we have $b_{n,2}=[\omega_{n}^{\frac{2}{n}}n(n-2)]^{\frac{n}{n-2}}$.
\begin{lem}
[Talenti's principle]If $u,v$ are solutions of the following equations,

\begin{align}
\begin{cases}
\Delta u(x)=f(x) & x\in\Omega,\\
u(x)=0 & x\in\partial\Omega,
\end{cases} &  & \begin{cases}
\Delta v(x)=f^{\#}(x) & x\in\Omega^{\#},\\
v(x)=0 & x\in\partial\Omega^{\#},
\end{cases}\label{eq:talenti's principle}
\end{align}
then we have that 
\[
v\geq u^{\#}.
\]
\end{lem}

We apply Talenti's principle to a $C^{2}$ function $u$ with support
in $\Omega$. By definition, $u^{\#}$ is non-increasing. Therefore,
if we fix $b>0$ and $-1<\beta<0$, we have
\begin{align}
\int_{\Omega}e^{bu^{\frac{n}{n-2}}}|x|^{n\beta}dx & \leq\int_{\Omega^{\#}}e^{b(u^{\#})^{\frac{n}{n-2}}}|x|^{n\beta}dx+C(\Omega)\nonumber \\
 & \leq\int_{\Omega^{\#}}e^{bv^{\frac{n}{n-2}}}|x|^{n\beta}dx+C(\Omega),\label{eq:Compare u and u^=000023}
\end{align}
where $v$ comes from (\ref{eq:talenti's principle}). Thus, in order
to prove the modified Moser-Trudiger-Adams inequality, we only have
to consider spherical symmetric domains and functions. 

We then state a lemma of Tarsi \cite{tarsi2012adams}.
\begin{lem}
\cite{tarsi2012adams}. \label{lem:.Tarsi's lemma-.}Let $p>1$. For
any $r>0$ there is a constant $C=C(p,r)$ such that for any positive
measurable function $f(s)$ on $(1,+\infty)$, satisfying 
\[
\int_{1}^{\infty}f^{p}s^{2p-1}ds\leq1
\]
then
\[
\int_{1}^{\infty}e^{rF^{q}(t)}\frac{dt}{t^{r+1}}\leq C
\]
where $\frac{1}{p}+\frac{1}{q}=1$, and 
\[
F(t)=\int_{1}^{t}\int_{\tau}^{\infty}f(s)dsd\tau.
\]
\end{lem}

We now prove the following Moser-Trudinger-Adams inequality in bounded
Euclidean domains, which first appears in \cite{Lam-Lu}. 
\begin{thm}
\label{thm:modified Adams-}Suppose that $\Omega\subset\mathbb{R}^{n}$
is a bounded domain and $0\in\Omega$. Suppose $u\in C_{c}^{2}(\Omega)$
and $-1<\beta<0$. There is a $C=C(\Omega)$ such that if $\|\Delta u\|_{n/2}\leq1$
then
\begin{equation}
\int_{\Omega}\exp(b_{n,2}(1+\beta)|u|^{\frac{n}{n-2}})|x|^{n\beta}dx\leq C.\label{eq:Modified Adams'}
\end{equation}
 The coefficient $b_{n,2}(1+\beta)$ here is sharp in the sense that
the inequality fails for bigger constant. In particular, $b_{4,2}=32\pi^{2}$.
\end{thm}

\begin{proof}
By (\ref{eq:Compare u and u^=000023}), we only have to consider radial
functions. Suppose that $v(x)$ is a $C^{2}$ radially decreasing
function with support in ball $B_{R}$ . Let 
\begin{equation}
w(t)=n^{\frac{2}{n}}\omega_{n}^{\frac{2}{n}}(n-2)^{2-\frac{2}{n}}v(Rt^{-1/(n-2)}).\label{eq:w(t);theorem 4.3}
\end{equation}
 We see that 
\begin{equation}
\int_{B_{R}}\exp(b_{n,2}\alpha|v|^{\frac{n}{n-2}})|x|^{n\beta}dx=\frac{n\omega_{n}R^{n\alpha}}{n-2}\int_{1}^{\infty}\exp\left(\frac{n\alpha}{n-2}|w|^{\frac{n}{n-2}}\right)\frac{dt}{t^{1+\frac{n\alpha}{n-2}}},\label{eq:connecting eq; Theorem 4.3}
\end{equation}
where $\alpha=\beta+1<1$. The condition $\int_{B_{R}}|\Delta v|^{n/2}dx\leq1$
is equivalent to 
\[
\int_{1}^{\infty}|w''(s)|^{\frac{n}{2}}s^{n-1}ds\leq1,
\]
where $w(t)=\int_{1}^{t}\int_{z}^{\infty}|w''(s)|dsdz$. Then, by
Lemma \ref{lem:.Tarsi's lemma-.}, 
\begin{equation}
\int_{1}^{\infty}\exp(r|w|^{\frac{n}{n-2}}(t))\frac{dt}{t^{r+1}}\leq C_{0},\label{eq: connecting eq2; Theorem 4.3}
\end{equation}
 where $C_{0}=C_{0}(n,r)$. Set $r=\frac{n\alpha}{n-2}$, then from
(\ref{eq:connecting eq; Theorem 4.3}) and (\ref{eq: connecting eq2; Theorem 4.3}),
we get
\begin{equation}
\int_{B_{R}}\exp(b_{n,2}\alpha|v|^{\frac{n}{n-2}})|x|^{n\beta}dx\leq C_{0}.\label{eq:connecting eq 3; theorem 4.3}
\end{equation}

Next, we prove the sharpness of the constant $b_{n,2}\alpha$. We
consider the following set: 
\[
\mathcal{B}=\left\{ b:\exists C_{0},\int_{B_{1}}e^{b|v|^{\frac{n}{n-2}}}|x|^{n\beta}dx\leq C_{0},\forall v\in C_{c}^{\infty}(B_{1})\ and\ \|\Delta v\|_{n/2}=1\right\} .
\]
Then, (\ref{eq:connecting eq 3; theorem 4.3}) implies $b_{n,2}\alpha\in\mathcal{B}$.
Let $u$ be a positive function with support in the unit ball $B_{1}$
and equals $1$ in $B_{r}$ with $0<r<1$. Let $b\in\mathcal{B}$,
then
\[
C_{0}\geq r^{n\alpha}\exp(b/\|\Delta u\|_{n/2}^{\frac{n}{n-2}}).
\]
This leads to
\begin{equation}
b\leq\alpha n\lim_{r\to0}C_{2,\frac{n}{2}}(B_{r};B_{1})^{\frac{2}{n-2}}\left(\log\frac{1}{r}\right),\label{eq:sharpness of bn2}
\end{equation}
where $C_{2,\frac{n}{2}}(B_{r};B_{1})=\inf\|\Delta u\|_{\frac{n}{2}}^{\frac{n}{2}}$
and the infimum is taken over all $u\in C_{0}^{\infty}(B)$ such that
$u=1$ on $B_{r}$. By \cite{adams1988sharp}, the right hand side
of (\ref{eq:sharpness of bn2}) is just $b_{n,2}\alpha$. Therefore,
$b_{n,2}\alpha=\sup\mathcal{B}$. 
\end{proof}
Our goal is to apply Theorem \ref{thm:modified Adams-} to obtain
estimates on conic 4-manifolds. In the rest of this section, we work
in dimension 4. Since Talenti's principle is also proved originally
in the form of Newton potential\cite{talenti1976elliptic}, from the
same argument in Theorem \ref{thm:modified Adams-}, the next corollary
is easily followed:
\begin{cor}
\label{cor:corollary for manifold adams}Suppose that $\Omega\subset\mathbb{R}^{4}$
is a bounded domain, $f\in C_{c}(\Omega)$ and $\|f\|_{L^{2}(\Omega)}\leq1$.
Let $u$ be the Newton potential of $f$, i.e.
\[
u=\int_{\Omega}\Gamma(x-y)f(y)dy,
\]
where $\Gamma(x)=\frac{1}{4\pi^{2}|x|^{2}}$. Then we have that 
\[
\int_{\Omega}\exp(32\pi^{2}(1+\beta)u^{2})|x|^{4\beta}dx\leq C.
\]
\end{cor}

Using the Green's function of the Laplacian on a 4-manifold, Corollary
\ref{cor:corollary for manifold adams} allows us to obtain estimates
in a neighborhood of a conic point. Following \cite{branson1992estimates},
we are able to prove the following global estimates:
\begin{cor}
\label{cor:Adams' ineq on conic mfds with Laplacian }Suppose that
$(M,g_{0},D,g_{D})$ is a conic 4-manifolds with $D=\sum_{i=1}^{k}p_{i}\beta_{i}$
and $\beta_{1}=\min\{\beta_{i}\}$. Let $g_{1}=g_{D}$ and $dV_{i}$
be the volume elements of $g_{i}$ , $i=0,1$. Let $\bar{u}=\fint_{M}udV_{1}$.
Then for any $u\in H^{2}(dV_{1})$
\[
\log\int_{M}\exp(4|u-\bar{u}|)dV_{1}\leq C+\frac{1}{8\pi^{2}(1+\beta_{1})}\|\Delta u\|^{2},
\]
 where $C=C(M)$ is a constant.
\end{cor}

\begin{proof}
Since the Green's function $G(x,y)$ of $\Delta$ exists on smooth
manifold $(M,g_{0})$, we have 
\[
u(x)-\tilde{u}=\int_{M}\Delta u(y)G(x,y)dV_{0}(y),
\]
where $\tilde{u}=\frac{1}{V_{0}(M)}\int_{M}u(y)dV_{0}(y)$. Note 
\begin{equation}
G(x,y)=\frac{1}{4\pi^{2}}r^{-2}g(r)+H(x,y)\label{eq:Green function expansion}
\end{equation}
 for $H$ a bounded function on $M\times M$, $r=d(x,y)$ and $g(r)$
a function with support in $B_{r}$ with radius smaller than the injective
radius of $(M,g_{0})$. First, suppose that 
\[
f(y)=\Delta u(y),\|f\|_{2}\leq1.
\]
Let $dV_{1}=\rho(x)dV_{0}=e^{4\gamma(x)}dV_{0}$ be the volume element
for the singular metric. Consider the following PDE on $M$
\begin{equation}
\Delta\psi(x)=\frac{1}{V_{1}(M)}\left(\rho(x)-\frac{V_{1}(M)}{V_{0}(M)}\right).\label{eq:add2}
\end{equation}
(\ref{eq:add2}) has a weak solution $\psi(x)\in W^{2,p}$ for $1<p<-\frac{1}{\beta_{1}}$.
By the Sobolev embedding theorem , $\psi(x)\in L^{\frac{4p}{4-2p}}(dV_{0})\subset L^{2}(dV_{0})$.
Let $\bar{u}=\frac{1}{V_{1}}\int_{M}u(y)dV_{1}(y)$. Then we have
\begin{equation}
u-\bar{u}=\int_{M}f(y)(G(x,y)-\psi(y))dV_{0}(y).\label{eq:add3}
\end{equation}
By $\holder's$ inequality, we have 
\begin{equation}
\left|\int_{M}f(y)\psi(y)dV_{0}\right|\leq\|f\|_{2}\|\psi\|_{2}.\label{eq:cauchy ineq showing mean values are equiv}
\end{equation}
Combining (\ref{eq:The solution}), (\ref{eq:cauchy ineq showing mean values are equiv})
and (\ref{eq:Green function expansion}), we have
\[
|u(x)-\bar{u}|\leq\left|\frac{1}{4\pi^{2}}\int_{B_{\delta}(x)}f(y)r^{-2}dV_{0}(y)\right|+C\|f\|_{2},
\]
where $\delta$ is the injective radius and $C=C(g_{1},g_{0})=\|\psi\|_{2}$.
Pick a normal coordinates around $x$. The metric $g_{ij}(y)=\delta_{ij}+O(|y|^{2})$.
Then we see that 
\begin{align}
\left|\frac{1}{4\pi^{2}}\int_{B_{\delta}(x)}f(y)r^{-2}dV_{0}(y)\right| & =\left|\frac{1}{4\pi^{2}}\int_{B_{\delta}(x)}f(y)r^{-2}(1+O(r^{2}))dy\right|\label{eq:f local estimates; cor4.5}\\
 & \leq\left|\frac{1}{4\pi^{2}}\int_{B_{\delta}(x)}f(y)|x-y|^{-2}dy\right|+C\|f\|_{2}.\nonumber 
\end{align}
Note that we may assume that $f$ has compact support in $B_{\delta}(x)$,
because the integral over the rest part of the manifolds can be controlled
by the $L^{2}$ norm of $f$. Let $u_{1}(x)=\frac{1}{4\pi^{2}}\int_{\mathbb{R}^{4}}f(y)|x-y|^{-2}dy$.
Then
\begin{equation}
|u-\bar{u}|\leq|u_{1}|+C\|\Delta u\|_{2}.\label{eq:u-baru<u1}
\end{equation}
Suppose that $x=p_{i}$ with index $\beta_{i}$. By Corollary \ref{cor:corollary for manifold adams},
we have 
\begin{equation}
\int_{B_{\delta}(x)}\exp(32\pi^{2}\alpha_{i}(u_{1})^{2})|z|^{4\beta_{i}}dz\leq c_{0},\label{eq:connecting eq;cor 4.5}
\end{equation}
 where $\alpha_{i}=(1+\beta_{i})$. If $v$ is a non-constant function
on $M$, by mean value inequality and (\ref{eq:connecting eq;cor 4.5})
we have 
\begin{align}
\int_{B_{\delta}(x)}\exp\left(4|v(z)|\right)|z|^{4\beta_{i}}dz & \leq\int_{B_{\delta}(x)}\exp\left(\frac{32\pi^{2}\alpha_{i}v^{2}}{\|\Delta v\|_{2}^{2}}+\frac{1}{8\pi^{2}\alpha_{i}}\|\Delta v\|_{2}^{2}\right)|z|^{4\beta_{i}}dz\nonumber \\
 & \leq c_{0}\exp\left(\frac{\|\Delta v\|_{2}^{2}}{8\pi^{2}\alpha}\right).\label{eq:rescale v=00003Dv/|v|_2}
\end{align}
Let $\alpha=1+\beta_{1}=\min\{1+\beta_{i}\}$. Combining (\ref{eq:u-baru<u1})
and (\ref{eq:rescale v=00003Dv/|v|_2}), we obtain in $B_{\delta}$;
\begin{equation}
\int_{B_{\delta}(x)}\exp(4|u(z)-\bar{u}|)\rho(z)dV_{0}\leq C\exp\left(\frac{\|\Delta u\|^{2})}{8\pi^{2}\alpha}\right).\label{eq:nnection eq2; cor 4.5}
\end{equation}
On $M-\cup_{i}B_{\delta}(p_{i})$, using partition of unity, we can
assume that $u(z)-\bar{u}$ vanishes in $B_{\delta/2}(p_{i})$ . Then
we can apply Adams' inequality in the form of \cite{branson1992estimates}
which gives 
\begin{align}
\int_{M-\cup B_{\delta/2}(p_{i})}\exp(4|u(z)-\bar{u}|)\rho(z)dV_{0} & \leq c_{0}\exp\left(\frac{\|\Delta u\|^{2})}{8\pi^{2}}\right)\label{eq:Standard BCY estimate; cor 4.5}\\
 & \leq c_{0}\exp\left(\frac{\|\Delta u\|^{2})}{8\pi^{2}\alpha}\right).\nonumber 
\end{align}
Thus, we combine (\ref{eq:connecting eq;cor 4.5}) and (\ref{eq:Standard BCY estimate; cor 4.5})
to get :
\[
\int_{M}\exp(4|u(z)-\bar{u}|)dV_{1}\leq C\exp\left(\frac{\|\Delta u\|^{2})}{8\pi^{2}\alpha}\right).
\]
This concludes the proof. 
\end{proof}
If the Paneitz operator $P$ is nonnegative with $Ker(P)=\{constants\}$,
we may define the pseudo differential operator $\sqrt{P}$ and the
Green's function of $\sqrt{P}$ has the same leading term as $-\Delta$.
See Lemma 1.6 in \cite{chang1995extremal} for details. Then we can
follow the proof in Corollary \ref{cor:Adams' ineq on conic mfds with Laplacian }
to derive a lower bound of ${\rm II}$ on conic manifolds:
\begin{thm}
\label{thm:Adams ineq on conic 4 mfd with paneitz}Suppose that $(M,g_{0},D,g_{D})$
is a conic 4-manifolds with $D=\sum_{i=1}^{k}p_{i}\beta_{i}$ and
$\beta_{1}=\min\{\beta_{i}\}$. Let $g_{1}=g_{D}$ and $dV_{i}$ be
the volume elements of $g_{i}$ , $i=0,1$. Let $\bar{u}=\fint_{M}u(y)dV_{1}(y)$.
Let $P$ be the Paneitz operator of $g_{0}$ on $M$. Suppose that
$P$ is nonnegative and $Ker(P)=\{constants\}$. Then for any $u\in H^{2}(dV_{1})$
\[
\log\int_{M}\exp(4|u(x)-\bar{u}|)dV_{1}(x)\leq C+\frac{1}{8\pi^{2}(1+\beta_{1})}\int_{M}u(x)Pu(x)dV_{0}(x),
\]
 where $C=C(\beta_{1},M)$ is a constant. 
\end{thm}

\section{Proof of Theorem \ref{thm:Subcritical}}

The equation of constant $Q$-curvature (\ref{eq:Constant Q equation})
is the Euler-Lagrange equation of the $\mathrm{II}$ functional, which
is studied in \cite{branson1992estimates} and \cite{chang1995extremal}
for smooth metrics. Recall that
\[
\mathrm{II}_{g}(u)=\langle P_{g}u,u\rangle_{g}+2\int_{M}QudV_{g}-\frac{k_{g}}{2}\log\fint_{M}\exp(4u)dV_{g},
\]
where $k_{g}=\int_{M}QdV_{g}$ and $\langle P_{g}u,u\rangle_{g}=\int_{M}uP_{g}udV_{g}$.
By a simple integration by part trick, $\langle P_{g}u,u\rangle_{g}$
may be well defined for $u\in H^{2}(dV_{g})$. It is obvious that
$\langle P_{g}u,u\rangle_{g}$ is conformally invariant. From now
on, we simply write $\langle Pu,u\rangle$ when no confusion arises. 

Let $(M^{4},g_{0},D,g_{D})$ be a conic manifold with $g_{D}(x)=e^{2\gamma(x)}g_{0}(x)$.
Let $g_{1}=g_{D}$ and $dV_{i}$ be the volume elements of $g_{i},\ i=0,1$.
The key estimate of \cite{branson1992estimates,chang1995extremal}
is to employ Adams' inequality \cite{adams1988sharp} to derive a
low bound of the ${\rm II}$ functional. By Adams' inequality and
its modified form\cite{branson1992estimates,chang1995extremal}, if
$u\in H^{2}(dV_{0})$ and $\langle Pu,u\rangle\leq1$, we have that
\begin{equation}
\int_{M}\exp\left(32\pi^{2}|u(x)-\tilde{u}|^{2}\right)dV_{0}(x)\leq c_{0}V_{0}(M),\label{eq:chang-yang Adams' ineq on mfd}
\end{equation}
where the mean value $\tilde{u}=\fint_{M}udV_{0}$. In conic 4-manifolds,
we use the modified Adams' inequality \ref{thm:Adams ineq on conic 4 mfd with paneitz}
to obtain the estimate for ${\rm II}$ functional. 
\begin{proof}
[Proof of Theorem \ref{thm:Subcritical}]By Theorem \ref{thm:Adams ineq on conic 4 mfd with paneitz}
\begin{equation}
\log\left(\int_{M}\exp(4|u-\bar{u}|)dV_{1}\right)\leq\frac{1}{8\pi^{2}\alpha}\langle Pu,u\rangle+C(\alpha,M),\label{eq:II estimate, them1 proof}
\end{equation}
where $\alpha=(1+\beta_{1})$ and $\bar{u}=\fint_{M}udV_{1}$. Mean
value inequality implies:
\begin{align}
\int_{M}Q_{1}udV_{1} & =\int_{M}Q_{1}(u-\bar{u})dV_{1}+\bar{u}k_{g_{1}}\label{eq:connecting eq; thm1 proof}\\
 & \leq\frac{1}{4\epsilon}\int_{M}Q_{1}^{2}dV_{1}+\epsilon\int_{M}\left(u-\bar{u}\right)^{2}dV_{1}+\bar{u}k_{g_{1}}.\nonumber 
\end{align}
Here $k_{g_{1}}=k_{g_{0}}+8\pi^{2}\sum_{i}\beta_{i}$ by Proposition
\ref{prop:(Gauss-Bonnet-Chern)-Suppose-tha}. Notice that 
\begin{equation}
\|u-\bar{u}\|_{L^{2}(dV_{1})}\leq\|u-\bar{u}\|_{H^{2}}\leq\langle Pu,u\rangle,\label{eq:connection eq2; thm1 proof}
\end{equation}
by Propositions \ref{prop:(Sobolev's-embedding)-There} and \ref{prop:(Poincar=0000E9's-inequality)-If}.
If $\bar{u}=0$, then (\ref{eq:II estimate, them1 proof}),(\ref{eq:connecting eq; thm1 proof})
and (\ref{eq:connection eq2; thm1 proof}) give us the desired estimate:
\begin{align}
\langle Pu,u\rangle+2\int_{M}Q_{1}udV_{1} & -\frac{k_{g_{1}}}{2}\log\left(\fint_{M}\exp(4u)dV_{1}\right)\nonumber \\
 & \geq\left(1-\frac{k_{g_{1}}}{16\pi^{2}\alpha}-\epsilon\right)\langle Pu,u\rangle+C(\alpha,\epsilon).\label{eq:Estimate of II in subcritical case}
\end{align}
 We have used Lemma \ref{cor:-for-.lemforregularity Q_rho} for the
integrability of $Q_{1}^{2}$ in (\ref{eq:Estimate of II in subcritical case}).
If $k_{g_{1}}<16\pi^{2}\alpha$ , we can always choose an $\epsilon$
small enough such that $\frac{k_{g_{1}}}{16\pi^{2}\alpha}+\epsilon<1$.
Then (\ref{eq:Estimate of II in subcritical case}) shows 
\begin{equation}
{\rm II}(u)\geq C'\langle Pu,u\rangle+C\geq C''.\label{eq:Estimate}
\end{equation}

Let 
\[
\Lambda=\inf\{{\rm II}(u):u\in H^{2}(dV_{1}),\int_{M}udV_{1}=0\}.
\]
Take a minimizing sequence of $\mathrm{II}$, namely $\{u_{i}\}_{i=1}^{\infty}$
such that $\mathrm{II}(u_{i})\to\Lambda$ as $i\to\infty$ and $\int_{M}u_{i}dV_{1}=0$
. By (\ref{eq:Estimate}) we see that $\|\Delta u_{i}\|_{2}^{2}$
is bounded. Hence, $u_{i}$ is bounded in $H^{2}(dV_{1})$ by $\poincare's$
inequality Proposition \ref{prop:(Poincar=0000E9's-inequality)-If}.
Replaced by a subsequence, we may assume that $u_{i}$ converges weakly
to some $w$ in $H^{2}(dV_{1})$ and strongly to the same $w$ in
$L^{2}(dV_{0})$ by the compactness of the embedding, cf Proposition
\ref{prop:(compact-embedding)-The}. Then, we claim that $w$ achieves
the infimum.

Claim: ${\rm II}(w)=\Lambda$.
\begin{proof}
[Proof of the claim]Since $u_{i}$ converges to $w$ weakly in $H^{2}(dV_{1})$,
we see that 
\[
\langle Pw,w\rangle\leq\liminf\langle Pu_{i},u_{i}\rangle
\]
 and 
\[
\int_{M}Q_{1}wdV_{1}=\lim_{i\to\infty}\int_{M}Q_{1}u_{i}dV_{1}.
\]
In order to control the last term of $\mathrm{II}$, we note that
\begin{align*}
|\exp(4u_{i})-\exp(4w)| & =|\int_{u_{i}}^{w}4\exp(4s)ds|\\
 & \leq4\exp(4|w|+4|u_{i}|)|w-u_{i}|,
\end{align*}
which leads to
\begin{align}
\int_{M}|\exp(4u_{i})-\exp(4w)|dV_{1} & \leq4\int_{M}\exp(4|w|+4|u_{i}|)|w-u_{i}|dV_{1}\label{eq:conncecting equation; proof of the claim; sec 5}\\
\quad & \leq4\left(\int_{M}\exp(4|4w|)\right)^{\frac{1}{4}}\left(\int_{M}\exp(4|4u_{i}|)\right)^{\frac{1}{4}}\nonumber \\
 & \quad\quad\cdot\left(\int_{M}|w-u_{i}|^{2}\right)^{\frac{1}{2}}.\nonumber 
\end{align}
By Adams' Inequality, $\left(\int_{M}\exp(4|4w|)\right)^{\frac{1}{4}}$
and $\left(\int_{M}\exp(4|4u_{i}|)\right)^{\frac{1}{4}}$ are bounded.
Then (\ref{eq:conncecting equation; proof of the claim; sec 5}) shows
that $\int_{M}|\exp(4u_{i})-\exp(4w)|dV_{1}\to0$ as $i\to\infty$.
Hence 
\[
\lim_{i\to\infty}\log(\fint_{M}\exp(4u_{i})dV_{1})=\log(\fint_{M}\exp(4w)dV_{1}).
\]
Therefore, we have that ${\rm II}(w)\leq\liminf{\rm II}(u_{i})=\Lambda$.
By the definition of $\Lambda$ , we have $\mathrm{II}(w)=\Lambda$.
\end{proof}
\end{proof}
We have established the existence of a minimizer $w$ of the functional
${\rm II}$. $w$ satisfies the corresponding Euler-Lagrange equation:
\begin{equation}
P_{g_{1}}w+Q_{1}=\frac{k_{g_{1}}e^{4w}}{2\fint_{M}e^{4w}dV_{1}}.\label{eq:add4}
\end{equation}
 Note that (\ref{eq:add4}) is equivalent to $Q_{w}=c$ where $Q_{w}$
is the $Q$-curvature of metric $g_{w}=e^{2w}g_{1}.$ Therefore, $w$
is a weak solution of 
\begin{equation}
P_{g_{1}}w+Q_{1}=c\cdot e^{4w}.\label{eq:Weak solution of minimizer}
\end{equation}

We next prove the regularity of the solution in \ref{eq:Weak solution of minimizer}.
Suppose we have a $H^{2}$ solution $w$ of $P_{g_{1}}w+Q_{1}=c\cdot e^{4w}$.
By a renormalization of the volume, we may assume that $c=\pm1$ or
$0$. We assume that $c=1$ and $w$ is a weak solution of $P_{g_{1}}w+Q_{1}=e^{4w}$.
The proofs for other cases are similar. Thus, for any $v\in H^{2}(dV_{1})$,
$w$ satisfies:
\begin{align*}
0 & =\int_{M}vP_{g_{1}}wdV_{1}+\int_{M}(Q_{1}-e^{4w})vdV_{1}\\
 & =\int_{M}vP_{g_{0}}w\rho dV_{0}+\int_{M}(Q_{1}-e^{4w})v\rho dV_{0},
\end{align*}
where $\rho(x)=\frac{dV_{1}}{dV_{0}}=e^{4\gamma(x)}$. Thus, $w$
is a weak solution of 
\begin{equation}
P_{g_{0}}w=e^{4w}\rho-Q_{1}\rho\label{eq:equationForWeakSolution}
\end{equation}
in $H^{2}(dV_{0})$. Let 
\[
h(x)=e^{4w(x)}\rho(x)-Q_{1}(x)\rho(x)-\mathrm{div}(A_{g_{0}}dw).
\]
By Corollary \ref{cor:-for-.lemforregularity Q_rho}, $Q_{1}\rho\in L^{p}(dV_{0})$
for $1<p<2$, $\rho\in L^{q_{1}}(dV_{0})$ for $1<q_{1}<\frac{1}{1-\alpha}$
and by Adams' inequality, $e^{4w}\in L^{p}(dV_{0})$ for any $p>1$.
This implies that $e^{4w}\rho\in L^{q_{1}}(dV_{0})$ for $1<q_{1}<\frac{1}{1-\alpha}$.
Thus $h\in L^{q}(dV_{0})$ for some $1<q<\min\{2,\frac{1}{1-\alpha}\}$.
Let $z(x)=\Delta w$. Then $\Delta z=h(x)$ in weak sense, i.e. 
\[
\int_{M}z\Delta vdV_{0}=\int_{M}hvdV_{0},
\]
for any $v(x)\in H^{2}(dV_{0})$. Now let $\Gamma(x,y)$ be the Green's
function for $\Delta$. Let 
\[
H(x)=\int_{M}h(y)\Gamma(x,y)dV_{0}.
\]
The regularity theory of elliptic equations \cite{gilbarg2015elliptic}
shows that $H(x)\in W^{2,q}(dV_{0})$ and $\Delta H(x)=h(x)$ $a.e.$.
Therefore, 
\[
z(x)=H(x)+\bar{z}\quad a.e.,
\]
where $\bar{z}=\int_{M}zdV_{0}$. We apply the regularity theory of
elliptic equations again to obtain $w\in W^{4,q}(dV_{0})$ . Since
$W^{4,q}(dV_{0})$ is the regular Sobolev space and $4q>4$, so we
embed the solution $w$ into $\holder$ spaces by Morrey's embeeding
theorem:
\begin{equation}
w\in\begin{aligned}C^{\tau},\  & \tau<\min\{2,4(1+\beta)\}.\end{aligned}
\label{eq:regularity claim}
\end{equation}

For any $x_{0}\in M$ and $x_{0}\not=p_{i},i=1,2,...,k$, there exists
a small neighborhood $B_{2\epsilon}(x_{0})$ of $x_{0}$ such that
$p_{i}\not\in B_{2\epsilon}(x_{0})$. By \cite{chang1995extremal},\cite{malchiodi2006compactness},
the Green's function $G(x,y)$ of $P_{g_{0}}$ with respect to $g_{0}$
exists and is smooth on $M\times M\backslash\{(x,x)\}$. Furthermore,
$G(x,y)$ and its derivatives have the following asymptotic properties\cite{malchiodi2006compactness}:
\begin{align}
|G(x,y)-\frac{1}{8\pi^{2}}\log\frac{1}{|x-y|}| & \leq C,\ x\not=y,\label{eq:Green's function expansion, sec 5}\\
|\nabla^{i}G(x,y)|\leq C_{i}\frac{1}{|x-y|^{i}}, & \ i=1,2,3.\nonumber 
\end{align}
Here $C,C_{i},i=1,2,3$ are some constants depend on $(M,g_{0})$.
The Green's function of $P_{g_{0}}$ gives the representation of $w$,
\[
w(x)-\bar{w}=\int_{M}G(x,y)P_{g_{0}}w(y)dV_{0}(y).
\]
Let $f(x)=e^{4w(x)}\rho(x)-Q_{1}(x)\rho(x)$. $f(x)$ is clearly integrable
and bounded in $B_{\epsilon}(x_{0})$, which implies that $w(x)\in C^{3}(B_{\epsilon}(x_{0}))$.
Since $h(x)\in C^{1}(M\backslash\{p_{i}\})$, we can apply the regularity
theory to $\Delta z(x)=h(x)$ to show that $z(x)\in C_{loc}^{2,\tau}(M\backslash\{p_{i}\})$,
which implies $w\in C_{loc}^{4,\tau}(M\backslash\{p_{i}\})$. The
standard bootstrapping technique then yields $w\in C^{\infty}(M\backslash\{p_{i}\})$. 
\begin{rem}
From the regularity argument, we can see the number $2$ in (\ref{eq:regularity claim})
is introduced by the $L^{2-\epsilon}$ integrability of $Q_{1}\rho$.
This term disappears when the original metric is conformally flat.
In other words, the solution is in $C^{\tau}(M)$ for any $\tau<4(1+\beta_{1})$
if the metric is conformally flat.
\end{rem}

\section{Radial Symmetric Solutions}

In this section, we consider radial symmetric solutions on conic 4
spheres with standard background metric. Let $x_{S},x_{N}$ be the
two antipodes on $S^{4}$ and $D=\beta_{0}x_{S}+\beta_{1}x_{N}$.
Let $\eta:\mathbb{R}^{4}\to S^{4}$ be the inverse of the stereographic
projection from north pole. Then 
\[
\eta^{*}(g_{0})=e^{2z(x)}\ ds^{2}
\]
where $g_{0}$ is the standard metric on sphere, $ds^{2}$ is the
Euclidean metric, and
\[
z(x)=\log\frac{2}{|x|^{2}+1}.
\]
Note $\eta^{-1}$ maps $x_{N}$ to infinity and $x_{S}$ to $0$.
Let $T=S^{3}\times\mathbb{R}$ be a cylinder. Suppose
\[
G:S^{3}\times\mathbb{R}\to\mathbb{R}^{4},G(w,t)=e^{t}w.
\]
The composition map $\eta\circ G:T\to S^{4}$ gives the standard cylindrical
coordinate. The Paneitz operator with product metric $g_{T}$ is the
following: 
\[
P_{T}=(\partial_{t}^{2}+\Delta_{S^{3}})^{2}-4\partial_{t}^{2}.
\]
A radial symmetric function on $\mathbb{R}^{4}$ depends only on $t$.
The constant $Q$-curvature equation is
\begin{equation}
v''''(t)-4v''(t)=c\cdot e^{4v}.\label{eq:Radial symmsetry eq; sec 6}
\end{equation}
We only consider the positive $Q$-curvature since $k_{g}=8\pi^{2}(2+\beta_{0}+\beta_{1})>0$.
By adding a constant, we may normalize (\ref{eq:Radial symmsetry eq; sec 6})
to the following
\begin{equation}
v''''(t)-4v''(t)=e^{4v(t)}.\label{eq:ODE to solve}
\end{equation}

If a symmetric conic metric has $D=\beta_{0}x_{S}+\beta_{1}x_{N}$,
the corresponding $v(t)$ must have linear growth at $\pm\infty$.
Thus, we have the following boundary conditions at $\pm\infty$: 
\begin{equation}
\lim_{t\to-\infty}v'(t)=1+\beta_{0},\lim_{t\to+\infty}v'(t)=-\beta_{1}-1.\label{eq: linear growth at infinity}
\end{equation}
We classify all solutions of (\ref{eq:ODE to solve}) satisfying (\ref{eq: linear growth at infinity}).

Define $x_{1}(t)=v'(t),\ x_{2}(t)=x_{1}'(t),\ x_{3}(t)=x_{2}'(t)-4x_{1}(t)$
and $x_{4}(t)=x_{3}'(t)$. From (\ref{eq:ODE to solve}), we get the
following system
\begin{equation}
\left\{ \begin{array}{cccccc}
x_{1}' & = &  & x_{2}\\
x_{2}' & = & 4x_{1} &  & +x_{3}\\
x_{3}' & = &  &  &  & x_{4}\\
x_{4}' & = &  &  &  & 4x_{1}x_{4}
\end{array}\right.\label{eq:system}
\end{equation}
Note that the positivity of $Q$-curvature implies that $x_{4}>0$
for all $t$. The corresponding solution for standard 4-sphere (differing
by a renormalization) in (\ref{eq:ODE to solve}) is given by

\begin{equation}
v(t)=-\log\cosh t+\frac{1}{4}\log6.\label{eq:standard sphere, radial}
\end{equation}
It corresponds to $X(t)=(x_{1},x_{2},x_{3},x_{4})^{T}$ with
\begin{equation}
X(t)=(-\tanh t,-\text{sech}^{2}t,2\tanh t\left(\text{sech}^{2}t+2\right),6\text{sech}^{4}t)^{T}.\label{eq:system for standard sphere, sec 6}
\end{equation}

First, we establish the first integral of (\ref{eq:system}).
\begin{prop}
We have the following first integral
\begin{align}
2x_{2}^{2}-8x_{1}^{2}-4x_{1}x_{3}+x_{4} & =c,\label{eq:first intergral 1}
\end{align}
or equivalently,
\begin{equation}
2x_{2}^{2}+\frac{1}{2}x_{3}^{2}-\frac{1}{2}(x'_{2})^{2}+x_{4}=c.\label{eq:first integral 2}
\end{equation}
where $c$ is a constant.
\end{prop}

\begin{proof}
Two formulae are equivalent. Multiply $v'(t)$ on both sides of (\ref{eq:ODE to solve})
and integrate by parts, we get (\ref{eq:first intergral 1}). To get
(\ref{eq:first integral 2}), simply apply $4x_{1}(t)=x_{2}'(t)-x_{3}(t)$
to (\ref{eq:first intergral 1}). 
\end{proof}
\begin{rem}
It is clear that all fixed points of (\ref{eq:system}) lie on a line
$(a,0,-4a,0)$ for $a\in\mathbb{R}$. The first integral (\ref{eq:first intergral 1})
hence indicates that if (\ref{eq:system}) has a bounded solution
then 
\[
\lim_{t\to\infty}|x_{1}(t)|^{2}=\lim_{t\to-\infty}|x_{1}(t)|^{2}=a^{2}.
\]
In other words, $\beta_{0}=\beta_{1}=|a|-1$ is a necessary condition
for the existence of a solution for (\ref{eq:ODE to solve}) with
(\ref{eq: linear growth at infinity}). This is also a special case
of the general Pohazaev idenity.
\end{rem}

Since (\ref{eq:system}) is invariant under the transformation $t\to t+c$,
we need to fix the gauge. First, we consider a special case with the
following initial data: 
\begin{equation}
x_{1}(0)=x_{3}(0)=0,x_{2}(0)=p,\ x_{4}(0)=q.\label{eq:initial condition sec 6}
\end{equation}
Here $p<0$ and $q>0$. Such solution is symmetric with respect to
$t=0$, i.e. $x_{1}(t)$ and $x_{3}(t)$ are odd functions while $x_{2}(t)$
and $x_{4}(t)$ are even functions. The constant $c$ in (\ref{eq:first intergral 1})
and (\ref{eq:first integral 2}) is given by $c=2p^{2}+q$. Note that
the standard solution of 4-sphere in (\ref{eq:standard sphere, radial})
has $p=-1$ and $q=6$.

For a fix $p<0$, we define 
\[
\mathcal{Q}=\{q>0:\ \forall t>0,\ x_{2}(t)<0\ with\ x_{2}(0)=p,\ x_{4}(0)=q\}.
\]
We show that $q=\sup\mathcal{Q}$ will give the precise initial data
in (\ref{eq:initial condition sec 6}) so that the corresponding solution
is desired. We first state some lemmas to show that $\mathcal{Q}$
is connected, nonempty, and bounded from above. 
\begin{lem}
(Monotonicity Lemma)\label{lem:ODE (Monotonicity-lemma)} If $x_{i}(t)$
and $y_{i}(t)$ are two solutions for the system (\ref{eq:system})
and $x_{i}(0)\geq y_{i}(0)$,$i=1,2,3,4$, then $x_{i}(t)\geq y_{i}(t)$
for all $t>0$. The equality holds if and only if $x_{i}(0)=y_{i}(0)$,$i=1,2,3,4$.
\end{lem}

\begin{proof}
When $x_{i}(0)=y_{i}(0)$, $i=1,2,3,4$, then obviously $x_{i}(t)\equiv y_{i}(t)$
for all $t$. Suppose that $x_{j}(0)>y_{j}(0)$, for some $j\in\{1,2,3,4\}$.
Then by continuity, there exits a $\epsilon>0$ such that $x_{i}(t)>y_{i}(t)$
for $t\in(0,\epsilon)$, $1\leq i\leq j.$ In particular, $x_{1}(t)>y_{1}(t)$
and 
\[
(\log x_{4}(t))'-(\log y_{4}(t))'=4\left(x_{1}(t)-y_{1}(t)\right)>0,
\]
for $t\in(0,\epsilon)$. Therefore, $x_{4}(t)>y_{4}(t)$ on $(0,\epsilon)$
and $x_{i}(t)>y_{i}(t)$ on $(0,\epsilon)$, for $i=1,2,3,4$. Let
\[
J=\{t>0:x_{i}(t)>y_{i}(t)\},t_{0}=\inf\{t>0:t\not\in J\}.
\]
If $t_{0}<\infty$, $x_{i}(t_{0})>y_{i}(t_{0})$ for $i=1,2,3$ and
$x_{4}(t_{0})=y_{4}(t_{0})$. However, 
\[
(\log x_{4})'(t_{0})-(\log y_{4})'(t_{0})=4(x_{1}(t_{0})-y_{1}(t_{0}))>0.
\]
It is impossible because of the definition of $J$. This implies that
$J=(0,\infty)$. 
\end{proof}
The next lemma shows that $\mathcal{Q}$ is not empty. 
\begin{lem}
If $4p+q\leq0$, then $x_{2}<0$ and $x_{2}(t)\to-\infty$ as $t\to\infty$.
\label{lem:4p+q}
\end{lem}

\begin{proof}
If $4p+q<0$, then $x_{2}(0)=-\frac{q}{4}<0$, $x_{2}'(0)=0$, and
$x_{2}''(0)=4p+q<0$. Hence $x_{2}(t)<0{\rm \ and\ }x_{2}'(t)<0$
for $t\in(0,\epsilon)$. Clearly, by (\ref{eq:system}),
\[
x_{2}'''(t)=4x_{2}'(t)+x_{4}'(t),
\]
 $x_{2}'''(t)$ is negative if $x_{2}'(t)<0$. Therefore $x_{2}'(t)\to-\infty$
and $x_{2}(t)\to-\infty$ as $t\to\infty$.

If $4p+q=0$, then we have $x_{2}^{(4)}(0)=x_{4}''(0)=4pq<0$. Thus,
$x_{2}'(t)<0$ for $t\in(0,\epsilon)$. We then follow the above argument
to get the same result.
\end{proof}
Then we prove the boundedness of $\mathcal{Q}$ by a comparison argument. 
\begin{lem}
For each $p<0$ , there is some $q>0$ such that $\exists T>0$, $x_{2}(t)>0$
for $t>T$.\label{lem:existsqx2positive}
\end{lem}

\begin{proof}
We argue by contradiction. Suppose that on the contrary, for any $q>0$
, $x_{2}(t)<0$ for all $t>0$. Then, $x_{1}(t)<0$ , $x_{4}'(t)<0$,
for $t\in(0,+\infty)$ and
\begin{align*}
x_{4}''(t) & =4(x_{2}(t)+4x_{1}^{2}(t))x_{4}(t)\geq4x_{2}(t)x_{4}(t).
\end{align*}
By (\ref{lem:4p+q}) we may assume that $4p+q>0$. Thus $x_{2}'(t)>0$
for $t\in(0,\epsilon)$ where $\epsilon$ is small. Let $T>0$ be
maximal such that $x_{2}'(t)\geq0$ for $t\in[0,T)$ . Then, $x_{4}(t)\leq q$
and $x_{2}(t)\geq p$ for $t\in[0,T)$. Hence, we obtain direct estimates:
\begin{align*}
x_{4}''(t) & \ge4x_{2}(t)x_{4}(t)\geq4pq,\\
x_{4}'(t) & \geq4pqt,\\
x_{4}(t) & \geq q+2pqt^{2}.
\end{align*}
 Since 
\[
x_{2}''(t)=4x_{2}+x_{4}\geq4p+q+2pqt^{2},
\]
 we have 
\begin{equation}
x_{2}'(t)\geq(4p+q)t+\frac{2}{3}pqt^{3}.\label{eq:estimateforx2'}
\end{equation}
(\ref{eq:estimateforx2'}) holds for all $t\in(0,T]$ especially for
$t=T$. By (\ref{eq:estimateforx2'}), we see that 
\[
T\geq\sqrt{\frac{4p+q}{-\frac{2}{3}pq}}=\sqrt{\frac{4\frac{p}{q}+1}{-\frac{2}{3}p}}.
\]
 Another consequence of (\ref{eq:estimateforx2'}) is 
\[
x_{2}(t)\geq p+\frac{t^{2}}{2}(4p+q)+\frac{1}{6}pqt^{4}.
\]
Let $t_{1}<t_{2}$ be two positive roots of $p+\frac{t^{2}}{2}(4p+q)+\frac{1}{6}pqt^{4}$.
If there exists $t<T$ such that 
\[
t_{1}^{2}=\frac{-\sqrt{\left(4p+q\right)^{2}-\frac{8p^{2}q}{3}}+\left(4p+q\right)}{-\frac{2}{3}pq}<t^{2}<\frac{\sqrt{\left(4p+q\right)^{2}-\frac{8p^{2}q}{3}}+\left(4p+q\right)}{-\frac{2}{3}pq}=t_{2}^{2},
\]
then $x_{2}(t)>0$, which contradicts with our assumption. Thus, it
is sufficient to show the interval $(t_{1},t_{2})\cap(0,T)\not=\varnothing$
for large $q$. Let $z=\frac{p}{q}$. We see that
\[
t_{1}^{2}=\frac{-\sqrt{\left(4z+1\right)^{2}-\frac{8pz}{3}}+\left(4z+1\right)}{-\frac{2}{3}p}\to0,\quad as\ q\to\infty,
\]
\[
\sqrt{\frac{4p+q}{-\frac{2}{3}pq}}=\sqrt{\frac{4z+1}{-\frac{2}{3}p}}\to\sqrt{-\frac{3}{2p}}>0,\quad as\ q\to\infty.
\]
Hence, for large $q$, 
\[
t_{1}<\sqrt{\frac{4p+q}{-\frac{2}{3}pq}}\leq\min\{T,t_{2}\}.
\]
We have thus finished the proof.
\end{proof}
Finally, we prove that there are bounded solutions of (\ref{eq:system}),
which establishes the existence part of Theorem \ref{thm:ODE}.
\begin{thm}
For any fixed $p<0$, there is a unique $q>0$ such that $4p+q>0$
and the solution of system (\ref{eq:system}) is bounded for all $t$
with initial data $x_{1}(0)=x_{3}(0)=0$, $x_{2}(0)=p$ and $x_{4}(0)=q.$
\label{thm:For-any-fixed6.6}
\end{thm}

\begin{proof}
Let 
\[
\mathcal{Q}=\{q>0:\ \forall t>0,\ x_{2}(t)<0\ with\ x_{2}(0)=p,\ x_{4}(0)=q\}.
\]
By Lemma \ref{lem:4p+q}, if $4p+q\leq0$, $q\in\mathcal{Q}\not=\varnothing$.
By monotonicity lemma \ref{lem:ODE (Monotonicity-lemma)}, $\mathcal{Q}$
is a connected set and by Lemma \ref{lem:existsqx2positive}, $q_{0}=\sup\{q\in\mathcal{Q}\}<\infty$.
We claim that $q_{0}$ is a choice such that the corresponding solution
of (\ref{eq:system}) is bounded. Let $\{y_{i}(t),i=1,2,3,4\}$ be
a solution of (\ref{eq:system}) with initial value:
\[
(y_{1},y_{2},y_{3},y_{4})(0)=(0,p,0,q_{0}).
\]
 We claim that $y_{2}(t)\to0$ as $t\to\infty$. We prove this claim
by excluding several cases.

Case 1. $\exists t_{0}>0$ such that $y_{2}(t)<0$ for $0\leq t<t_{0}$,
$y_{2}(t_{0})=0$, and $y_{2}'(t_{0})>0$. Then, there is $t_{1}>t_{0}$
such that $y_{2}(t_{1})>0$. Because our solutions are continuously
dependent on the initial values, there is a $q'<q_{0}$ such that
a solution $z(t)$ with $z_{i}(0)=y_{i}(0),i=1,2,3$, $z_{4}(0)=q'$
and $|z_{2}(t_{1})-y_{2}(t_{1})|<y_{2}(t_{1})/2$. Then $z(t_{1})>\frac{y_{2}(t_{1})}{2}>0$
which contradicts the definition of $q_{0}$.

Case 2. $\exists t_{0}>0$ such that $y_{2}(t)<0$ for $0\leq t<t_{0}$,
$y_{2}(t_{0})=0$, and $y_{2}'(t_{0})=0$. This can be ruled out since
\[
y_{2}''(t_{0})=4y_{2}(t_{0})+y_{4}(t_{0})>0,
\]
and $y_{2}(t_{0})$ is a local minimum, which contradicts the assumption
of $t_{0}$. 

We conclude from the Case 1 and 2 that $y_{2}(t)<0$ for all $t>0$.
Hence, $y_{1}(t)<0$ for all $t>0$.

Case 3. $\liminf\limits _{t\to+\infty}y_{2}(t)=-\infty$. 

Pick an increasing sequence $\{t_{k}\}$ such that $t_{k}\to\infty$
and $y_{2}(t_{k})\to-\infty$ as $k\to+\infty$. For each $k$, we
assume further that $y_{2}(t)<-\varepsilon_{k}$ for some $\varepsilon_{k}$
on $(0,t_{k})$. By the definition of $q_{0}$, there is a sequence
$\{q_{i}\}$ such that $q_{i}>q_{0}$ and $q_{i}\to q_{0}$, and there
is a sequence of solutions $\{x^{i}(t)\}_{i=1}^{\infty}$ with initial
value $(0,p,0,q_{i})$ such that 
\[
||x_{j}^{i}(t)-y_{j}(t)||_{\infty}\to0,\ j=1,2,3,4,
\]
as $i\to\infty$ in any compact subset of $\mathbb{R}$. For $t\in(0,t_{k}]$,
pick $i_{k}$ such that $||x_{j}^{i_{k}}(t)-y_{j}(t)||_{\infty}<\varepsilon_{k}$.
However, since $q_{i_{k}}>q_{0}$, there is $t_{i_{k}}^{*}$ such
that $x_{2}^{i_{k}}(t_{i_{k}}^{*})=0$. By mean value theorem, there
is a $\tau_{i_{k}}>0$ such that $(x_{2}^{i_{k}})'(\tau_{i_{k}})=0$
and $x_{2}^{i_{k}}(\tau_{i_{k}})\leq y_{2}(t_{k})$. By (\ref{eq:first integral 2}),
\begin{align}
(y_{2}(t_{k}))^{2} & \leq2(x_{2}^{i_{k}})^{2}-\frac{1}{2}((x_{2}^{i_{k}})')^{2}+\frac{1}{2}(x_{3}^{i_{k}})^{2}+x_{4}^{i_{k}}\label{eq:y_2 is bounded contradiction}\\
 & =2p^{2}+q_{i_{k}}<2p^{2}+2q_{0}.\nonumber 
\end{align}
This contradicts with our assumption of Case 3.

Case 4. $\lim\limits _{t\to+\infty}y_{2}(t)=-c<0$.

In this case, $y_{1}(t)<-\frac{c}{2}t+b$ for some constant $b$.
By (\ref{eq:system}), 
\[
y_{4}'(t)<(b-\frac{c}{2}t)y_{4},
\]
 and 
\[
y_{4}(t)\leq C\exp(bt-\frac{c}{4}t^{2}).
\]
Therefore, $y_{3}(t)=\int_{0}^{t}y_{4}(s)ds+y_{3}(0)$ is bounded.
Since $y_{2}'(t)=4y_{1}+y_{3}$ , we see $y_{2}'(t)<-\frac{c}{2}t+b_{1}$
for some constant $b_{1}$, which implies that $y_{2}$ is not bounded
and cannot have negative limit. We have reached a contradiction with
the assumption of Case 4.

Case 5. $\liminf\limits _{t\to+\infty}y_{2}(t)=-c<0$ while $\limsup\limits _{t\to+\infty}y_{2}(t)>-c$. 

We pick a sequence $t_{n}\to+\infty$ such that $y_{2}(t_{n})\to-c$
, $y_{2}'(t_{n})=0$, 
\[
y_{2}''(t_{n})=4y_{2}(t_{n})+y_{4}(t_{n})\geq0,
\]
 which means $\lim_{n\to\infty}y_{4}(t_{n})\geq4c$. But by (\ref{eq:system}),
\[
y_{4}'(t)=4y_{4}(t)y_{1}(t)<0
\]
 implies that $y_{4}$ is monotone and $y_{4}(t)\geq4c$ for $t\to\infty$.
However, by (\ref{eq:system}), $y_{3}(t)$ is then unbounded. By
evaluating (\ref{eq:first intergral 1}) at $t_{n}$, a contradiction
is reached.

We summarize our discussion above. By ruling out cases 1 to 5, we
have proved that $y_{2}(t)\to0$ as $t\to+\infty$. 

We now prove that $y(t)$ is a bounded solution of (\ref{eq:system}).
We split the proof into two cases. 

Case 1, suppose that $y_{2}(t)$ oscillates as $t\to+\infty$, i.e.
there exist $t_{k}\to\infty$ as $k\to\infty$ such that $y_{2}'(t_{k})=0$. 

Since $y_{3}$ is monotone, (\ref{eq:first integral 2}) implies that
$y_{3}$ is bounded. Furthermore, by (\ref{eq:first integral 2}),
$y_{2}'(t)$ is bounded. Then, $y_{2}'(t)=4y_{1}(t)+y_{3}(t)$ implies
that $y_{1}(t)$ is bounded. Hence, we have proved the theorem for
this case.

Case 2, if $y_{2}$ is increasing for big $t$, i.e. $y_{2}'>0$ for
$t>t^{*}\gg0$. 

By (\ref{eq:system}), we have
\[
\frac{d}{dt}\left(\frac{1}{2}(y_{2}'(t))^{2}-2y_{2}^{2}(t)\right)>0,t>t^{*}.
\]
Thus, for $t^{*}<t_{1}<t_{2}$.
\begin{equation}
\frac{1}{2}(y_{2}'(t_{2}))^{2}-2y_{2}^{2}(t_{2})>\frac{1}{2}(y_{2}'(t_{1}))^{2}-2y_{2}^{2}(t_{1}),\label{eq:monotonicity in the last part sec 6}
\end{equation}
We claim $\frac{1}{2}(y_{2}'(t_{1}))^{2}-2y_{2}^{2}(t_{1})\leq0$
for $t_{1}>t^{*}$. If not, then for any $t_{1}>t^{*}$, 
\begin{equation}
\frac{1}{2}(y_{2}'(t_{1}))^{2}-2y_{2}^{2}(t_{1})>c>0.\label{monotonicity formula in the last part of proof og them6.6; sec6}
\end{equation}
Since $y_{2}(t)\to0$, we may choose $t_{2}>t_{1}$ such that $|y_{2}(t)|<\frac{c}{2}$.
Then, by (\ref{eq:monotonicity in the last part sec 6}) and (\ref{monotonicity formula in the last part of proof og them6.6; sec6}),
we have $\frac{1}{2}(y_{2}'(t_{2}))^{2}>\frac{c}{2}$ . This shows
that $y_{2}$ is at least linearly increasing, which contradicts with
the fact that $y_{2}\to0$. Hence, we have proved that
\begin{equation}
\frac{1}{2}(y_{2}'(t_{1}))^{2}-2y_{2}^{2}(t_{1})\leq0,\label{eq:last bounded;sec 6}
\end{equation}
for $t_{1}>t^{*}$. (\ref{eq:last bounded;sec 6}) shows that $y_{2}'$
and $y_{3}(t)$ are bounded. We use $y_{2}'(t)=4y_{1}(t)+y_{3}(t)$
in (\ref{eq:system}) to show that $y_{1}(t)$ is also bounded. This
finishes the proof of case 2. 

We have thus proved Theorem \ref{thm:For-any-fixed6.6}.
\end{proof}
We now discuss the uniqueness part in Theorem \ref{thm:ODE}.
\begin{thm}
Fix a constant in the right hand side of the first integral (\ref{eq:first intergral 1}),
the bounded solution to the system (\ref{eq:system}) is unique up
to a translation(dilation) in $t$.
\end{thm}

\begin{proof}
By a translation, we may assume that $x_{1}(0)=0$. Suppose the initial
data is given by 
\[
(x_{1},x_{2},x_{3},x_{4})(0)=(0,a,b,c),
\]
where $c>0$ and the corresponding solution $z_{i}(t)$. Then $(-1)^{i}z_{i}(-t)$
with initial value $(0,a,-b,c)$ is also a solution of (\ref{eq:system}).
Without loss of generality, we may assume $b>0$ and $2a^{2}+c=2p^{2}+q$.
Suppose $y_{i}(t)$ is a bounded solution such that 
\[
(y_{1},y_{2},y_{3},y_{4})(0)=(0,p,0,q).
\]
If $p\leq a<0$ then $q\leq c$. By Lemma \ref{lem:ODE (Monotonicity-lemma)},$z_{i}(t)>y_{i}(t)$
for $t>0$. Hence 
\begin{equation}
z_{2}(t)-y_{2}(t)>0,(z_{2}(t)-y_{2}(t))'>0\label{eq:connecting inequality in theorem 6.7}
\end{equation}
Furthermore, since the inequalities in (\ref{eq:connecting inequality in theorem 6.7})
are strict, by Lemma \ref{lem:ODE (Monotonicity-lemma)}, if $y_{2}$
is bounded, then $z_{1}(t)\to\infty$. Thus, $z(t)$ can not be a
bounded solution of (\ref{eq:system}). If $p>a$ then $q>c$, let
$\bar{z}_{i}(t)=(-1)^{i}z_{i}(-t)$. Then still $\bar{z}_{i}(t)<y_{i}(t)$
for $t>0$ . By the same argument, $\bar{z}_{1}(t)\to-\infty$ as
$t\to\infty$. Hence, $z(t)$ can not be a bounded solution either.
Thus, we have proved the theorem.
\end{proof}

\section{Asymptotic Behavior}

In this section, we establish a local asymptotic expansion for solutions
of (\ref{eq: PDE in R^4}) in $\mathbb{R}^{4}$. Let $\mathcal{P}_{m}$
be the space of homogeneous polynomials with degree $m$. The eigenvalues
and eigenfunctions of Laplacian $\Delta$ on $\mathbb{R}^{4}$ are
described as follows.
\begin{lem}
(e.g. \cite{lee1987yamabe}) Suppose that $x\in\mathbb{R}^{4}$ $r=|x|$.
The eigenvalues of $r^{2}\Delta$ on $\mathcal{P}_{m}$ are \label{lem:(e.g.-)-Lee-Park}
\[
\{\lambda_{j}=2j(2+2m-2j):\ j=0,1...,[m/2]\}.
\]
The eigenfunctions corresponding to $\lambda_{j}$ are the functions
of the form $r^{2j}u$, where $u\in\mathcal{P}_{m-2j}$ is harmonic.
\end{lem}

The above lemma indicates that if $a$ is not an eigenvalue, $(r^{2}\Delta-a)$
is invertible. The next lemma shows that with the presence of a singular
weight, we may still solve the double Laplace equation in $\mathcal{P}_{m}$.
\renewcommand{\labelenumi}{\arabic{enumi}.}
\begin{lem}
Let $\beta\in\mathbb{R}$, $-1<\beta<0$ . \label{lem:polynomial lemma}
\begin{enumerate}
\item $\beta\not=-1/2$. For any polynomial $f(x)$, there is a polynomial
$q(x)$ such that 
\[
\Delta^{2}(q(x)r^{4\beta+4})=f(x)r^{4\beta}.
\]
 More generally, there is a collection of polynomials $\{q_{l}\}_{l=0}^{k}$
such that 
\begin{equation}
\Delta^{2}(\sum_{l=0}^{k}q_{l}(x)r^{4\beta+4}(\log r)^{l})=f(x)r^{4\beta}(\log r)^{k}.\label{eq:log^l}
\end{equation}
\item $\beta=-1/2$. For any polynomial $f$ with degree $\leq2$ with $f(x)=\sum a_{ij}x_{i}x_{j}+\sum b_{i}x_{i}+c$,
there exists a function 
\begin{equation}
q(x)=a_{0}r^{2}+\tilde{a}_{ij}x_{i}x_{j}\log r+\tilde{b_{i}}x_{i}\log r+\tilde{c}\log r,\label{eq:form beta=00003D-1/2}
\end{equation}
 such that $\Delta^{2}(q(x)r^{2})=f(x)r^{-2}$. In particular, $a_{0}$
vanishes if $\tilde{a}_{ii}=0$, $i=1,2,3,4$, and $\tilde{a}_{ij}=0$
if $\deg f<2$. There exist polynomials $q_{l}$ such that 
\begin{equation}
\Delta^{2}\left(r^{2}\sum_{l=0}^{k+1}q_{l}(x)(\log r)^{l}\right)=f(x)r^{-2}(\log r)^{k}.\label{eq:beta=00003D-1/2 and (logr)^k}
\end{equation}
\end{enumerate}
\end{lem}

\begin{proof}
See the appendix.
\end{proof}
Suppose that $u$ is a desired weak solution of (\ref{eq: PDE in R^4}).
We consider
\[
\Delta^{2}u(x)=e^{4u(x)},\ \ x\in B_{1}(0).
\]
Then by the regularity theory in the proof of Theorem \ref{thm:Subcritical},
$u-\beta\log r\in C^{\infty}(B_{1}(0)-\{0\})\cap C^{4\beta+4-\epsilon}(B_{1}(0))$
for any $\epsilon>0$. Here $\beta=\beta_{0}$. Let $w=u(x)-\beta\log r$.
Then
\begin{equation}
\Delta^{2}w(x)=e^{4w(x)}|x|^{4\beta}.\label{eq:PDE in Asymptotic Expansion}
\end{equation}

\begin{thm}
\label{thm:Asymptotic expansion at singular points}Suppose that $w$
is a solution of (\ref{eq:PDE in Asymptotic Expansion}) in $B_{1}$
and $w\in C^{\infty}(B_{1}(0)-\{0\})\cap C^{4+4\beta-\epsilon}(B_{1}(0)),\forall\epsilon>0$.
If 
\begin{description}
\item [{Case$\ $1}] $-\frac{k+1}{k+2}<\beta<-\frac{k}{k+1}$, for some
$k=0,1,2,\cdots,$ then

\[
w=\sum_{l=1}^{k+1}q_{l}(x)r^{4l(\beta+1)}+\psi(x),
\]
where $\psi(x)\in C^{4,\gamma}$ and $q_{l}$ are polynomials.
\item [{Case$\ $2}] $\beta=-\frac{2k}{2k+1}$, for some $k=0,1,2,\cdots,$
then

\[
w=\sum_{l=1}^{2k+2}q_{l}(x)r^{4l(\beta+1)}+\psi(x),
\]
where $\psi(x)\in C^{4,\gamma}$ and $q_{l}$ are polynomials. 
\item [{Case$\ $3}] $\beta=-\frac{2k-1}{2k}$, for some $k=0,1,2,\cdots,$
then
\[
w=\sum_{l=1}^{2k}q_{l}(x)r^{4l(\beta+1)}P_{l}(\log r)+\psi(x).
\]
where $\psi(x)\in C^{4,\gamma}$ , $q_{l}$ and $P_{l}$ are polynomials. 
\end{description}
\end{thm}

We break the proof into 3 cases according to the value of $\beta$.
\begin{proof}
[Proof of Case 1,] $-\frac{k+1}{k+2}<\beta<-\frac{k}{k+1}$. Since
$w\in C^{4\beta+4-\epsilon}$, there exists a polynomial $g_{0}$
with degree not exceeding $3$ such that $\bar{w}:=w-g_{0}=o(r^{4\beta+4-\epsilon})$.
Furthermore,
\begin{equation}
\Delta^{2}\bar{w}=e^{4g_{0}}(e^{4\bar{w}}-1)r^{4\beta}+e^{4g_{0}}r^{4\beta}.\label{eq:=00005Cbarw; proof of case 1}
\end{equation}
Since $e^{4g_{0}}$ is smooth, we find a polynomial $\phi_{0}(x)$
such that 
\begin{equation}
(\phi_{0}(x)-e^{4g_{0}(x)})r^{4\beta}=O(r^{\gamma}),\label{eq:connecting eq ; proof case 1}
\end{equation}
 where $\gamma>0$. In fact, since $\beta\not=-\frac{1}{2}$, we pick
$\gamma=4\beta+4$. By Lemma \ref{lem:polynomial lemma}, there is
a polynomial $q_{0}(x)$ such that 
\begin{equation}
\Delta^{2}(q_{0}(x)r^{4\beta+4})=\phi_{0}(x)r^{4\beta}.\label{eq: finding q_0}
\end{equation}
Then, from (\ref{eq:=00005Cbarw; proof of case 1}) ,(\ref{eq:connecting eq ; proof case 1})
and (\ref{eq: finding q_0}), we have 
\begin{align}
\Delta^{2}(\bar{w}-q_{0}(x)r^{4\beta+4}) & =e^{4g_{0}}(e^{4\bar{w}}-1)r^{4\beta}+\left(e^{4g_{0}}-\phi_{0}\right)r^{4\beta}\label{eq:asymp proof bar w_0}\\
 & =o(r^{8\beta+4-\epsilon})+O(r^{\gamma}).\nonumber 
\end{align}
If $\beta>-\frac{1}{2}$, then the right hand side of (\ref{eq:asymp proof bar w_0})
is $\holder$ continuous and by the elliptic regularity theory, we
have finished the proof.

If $\beta<-\frac{1}{2}$, let $w_{1}=\bar{w}-q_{0}(x)r^{4\beta+4}.$
Then 
\begin{equation}
\Delta^{2}w_{1}\in L^{p},1<p<-\frac{4}{4+8\beta}.\label{eq: w_1; proof of case 1}
\end{equation}
By the elliptic regularity theory, $w_{1}\in W^{4,p}\hookrightarrow C^{8\beta+8-\epsilon}$.
 Then we start the iteration procedure. At each step $m\leq j\leq k$,
suppose that there are functions $w_{m}$, $\bar{w}_{m}$, $s_{m}$,
$Q_{m}$, $g_{m}$, $R_{m}$, $\xi_{m}$, $\phi_{m}^{l}(x)$ such
that:\renewcommand{\labelenumi}{\roman{enumi}).}
\begin{enumerate}
\item $w_{m}\in C^{4(m+1)(\beta+1)-\epsilon}$ admits the following expansion
\begin{align}
w_{m} & =\bar{w}_{m-1}-\sum_{l=0}^{k}q_{m-1}^{l}r^{4(l+1)(\beta+1)}\nonumber \\
 & =\bar{w}_{m-1}-Q_{m-1}(x),\label{w_j expression =00005Cbetanot=00003D-1/2}
\end{align}
where $Q_{m-1}(x)=\sum_{l=0}^{k}q_{m-1}^{l}r^{4(l+1)(\beta+1)}$ and
$s_{m}=s_{m-1}+Q_{m-1}$ , $s_{0}=0$.
\item $g_{m}$ are polynomials with $deg(g_{m})\leq3$, such that 
\[
\bar{w}_{m}:=w_{m}-g_{m}=o(r^{4(m+1)(\beta+1)-\epsilon}).
\]
\item $\bar{w}_{m}$ satisfies: 
\begin{align}
\Delta^{2}\bar{w}_{m} & =e^{4s_{m-1}}(e^{4(\bar{w}_{m}+g_{m}+\sum_{l=0}^{k}q_{m-1}^{l}r^{4(l+1)(\beta+1)})}-1)r^{4\beta}+\xi_{m}(x)\label{eq:=00005CDelta^2 w_j =00005Cbeta=00005Cnpt=00003D-1/2}\\
 & =e^{4s_{m}}(e^{4\bar{w}_{m}}-1)r^{4\beta}+e^{4s_{m-1}}[e^{4(g_{m}+\sum_{l=0}^{k}q_{m-1}^{l}r^{4(l+1)(\beta+1)})}-1]e^{4\beta}+\xi_{m}(x)\nonumber \\
 & =o(r^{4[(m+2)\beta+m+1]-\epsilon})+R_{m}(x)+\xi_{m}(x).\label{eq:=00005CDelta^2 w_j remaining terms}
\end{align}
where
\begin{align}
R_{m}(x) & =e^{4s_{m-1}}[e^{4(g_{m}+\sum_{l=0}^{k}q_{m-1}^{l}r^{4(l+1)(\beta+1)})}-1]e^{4\beta},\label{eq:R_m}
\end{align}
and $\xi_{m}(x)$ is a $\holder$ continuous functions. We observe
that $R_{m}(x)$ can be expanded as
\begin{equation}
R_{m}(x)=\sum_{l=0}^{k}\phi_{m}^{l}(x)r^{4[(l+1)\beta+l]}+O(r^{4[(k+2)\beta+k+1]}),\label{eq:R_m expansion}
\end{equation}
where $\phi_{m}^{l}(x)$ are polynomials.
\end{enumerate}
At step $j\leq k$, by Lemma \ref{lem:polynomial lemma}, for each
$l$, there exist a polynomial $q_{j}^{l}(x)$ such that 
\[
\Delta^{2}(q_{j}^{l}(x)r^{4(l+1)(\beta+1)})=\phi_{j}^{l}(x)r^{4[(l+1)\beta+l]}.
\]
Let $Q_{j}(x)=\sum_{l}q_{j}^{l}r^{4(l+1)(\beta+1)}$ and $w_{j+1}=\bar{w}_{j}(x)-Q_{j}(x)$.
Then by (\ref{eq:=00005CDelta^2 w_j =00005Cbeta=00005Cnpt=00003D-1/2})
and (\ref{eq:R_m expansion})
\begin{align}
\Delta^{2}w_{j+1} & =o(r^{4[(j+2)\beta+j+1]-\epsilon})+R_{j}(x)-\sum_{l=0}^{k}\phi_{j}^{l}r^{4[(l+1)\beta+l]}+\xi_{j}\label{eq:w_j+1}\\
 & =o(r^{4[(j+2)\beta+j+1]-\epsilon})+O(r^{4[(k+2)\beta+k+1]})+\xi_{j}\nonumber \\
 & =o(r^{4[(j+2)\beta+j+1]-\epsilon})+\xi_{j+1},\nonumber 
\end{align}
where $\xi_{j+1}(x)$ is a $\holder$ continuous function. Let $s_{j+1}=s_{j}+Q_{j}$.
If $j<k$, the right hand side of (\ref{eq:w_j+1}) is in $L^{p}$
for $1<p<-\frac{1}{(j+2)\beta+j+1}$. Thus $w_{j+1}\in W^{4,p}$ by
the elliptic regularity theory and hence $w_{j+1}\in C^{4(j+2)(\beta+1)-\epsilon}$.
Therefore, there exists a corresponding $g_{j+1}$ such that 
\[
\bar{w}_{j+1}:=w_{j+1}-g_{j+1}=o(r^{4(j+2)(\beta+1)-\epsilon}).
\]
Note $g_{j+1}$ has degree less than or equal to 3. Then, by (\ref{eq:=00005CDelta^2 w_j =00005Cbeta=00005Cnpt=00003D-1/2}),
\begin{align}
\Delta^{2}\bar{w}_{j+1} & =\Delta^{2}w_{j+1}=\Delta^{2}(\bar{w}_{j}-Q_{j})\label{eq:=00005Cbarw_j+1 ,revised the induction part, sec 7}\\
 & =e^{4s_{j}}(e^{4\bar{w}_{j}}-1)r^{4\beta}+e^{4s_{j-1}}(e^{4(g_{j}+Q_{j-1})}-1)r^{4\beta}+\xi_{j}(x)-\sum_{l=0}^{k}\phi_{j}^{l}(x)r^{4[(l+1)\beta+l]}\nonumber \\
 & =e^{4s_{j}}(e^{4\bar{w}_{j}}-1)r^{4\beta}+\left(R_{j}-\sum_{l=0}^{k}\phi_{j}^{l}(x)r^{4[(l+1)\beta+l]}\right)r^{4\beta}+\xi_{j}\nonumber \\
 & =e^{4s_{j+1}}(e^{4\bar{w}_{j+1}}-1)r^{4\beta}+e^{4s_{j}}(e^{4(g_{j+1}+Q_{j})}-1)r^{4\beta}+O(r^{4[(k+2)\beta+k+1]})+\xi_{j}\nonumber \\
 & =e^{4s_{j+1}}(e^{4\bar{w}_{j+1}}-1)r^{4\beta}+R_{j+1}+\xi_{j+1}\nonumber \\
 & =o(r^{4[(j+3)\beta+j+2]-\epsilon})+R_{j+1}+\xi_{j+1}.\nonumber 
\end{align}
By expanding $\bar{w}_{j+1}$ as in (\ref{eq:=00005Cbarw_j+1 ,revised the induction part, sec 7}),
$R_{j+1}(x)$ is in the form of (\ref{eq:R_m}). If $j<k$, then we
apply the above argument again for $m=j+1$. We iterate until $j=k$
when the right hand side of (\ref{eq:w_j+1}) is $\holder$ continuous.
Thus, by the elliptic regularity theory, $w_{k+1}\in C^{4,\gamma}$
for some $\gamma>0$. This clearly gives the asymptotic expansion.
\end{proof}
\begin{proof}
[Proof of Case 2,] $\beta=\frac{1}{2k+1}-1$. This case is similar
to Case 1 with minor changes. We assume that $k\geq1$. The beginning
steps are the same as in the first case. We use iteration. In this
case, we assume that $w_{1}\in C^{8\beta+8}$. If $\beta=\frac{1}{2k}-1$,
the remaining term like (\ref{eq:=00005CDelta^2 w_j remaining terms})
has the following form
\[
R_{1}(x)=\phi_{0}r^{4\beta}+\phi_{1}r^{8\beta+4}+\cdots+\phi_{2k-1}r^{8k\beta+8k-4}+\phi_{2k}+O(r^{\gamma})
\]
where $\phi_{l}$ are polynomials. 
\[
4(l+1)(\beta+1)-4\not=-2,\ l=0,1,\cdots,2k.
\]
So by Lemma \ref{lem:polynomial lemma}, there exist polynomials $q_{1}^{l}(x)$
and $Q_{1}(x)=\sum_{l}q_{1}^{l}(x)r^{4(l+1)(\beta+1)}$ such that
\[
\Delta^{2}Q_{1}(x)=\sum_{l=0}^{2k}\phi_{l}(x)r^{4(l+1)\beta+4l}.
\]
The iteration procedures are almost the same as in Case 1. Note, however,
at the final step, when $j=2k-1$. By (\ref{eq:=00005Cbarw_j+1 ,revised the induction part, sec 7}),
\[
\Delta^{2}w_{2k}=o(r^{-\epsilon})+\xi_{2k}(x),
\]
and $w_{2k}\in C^{4-\epsilon}$. Then there exists a degree 3 polynomial
$g_{2k}$ such that $\bar{w}_{2k}=w_{2k}-g_{2k}\in C^{4-\epsilon}$.
Note that 
\begin{align}
\Delta^{2}\bar{w}_{2k} & =e^{4s_{2k-1}}\left(\exp\left(4\bar{w}_{2k}+g_{2k}+Q_{2k-1}\right)-1\right)r^{4\beta}+\xi_{2k-1}\label{eq beta =00003D1/2k+1 -1, w_2k}\\
 & =e^{4s_{2k}}(e^{4\bar{w}_{2k}}-1)r^{4\beta}+r^{4\beta}e^{4s_{1}}\left(e^{Q_{2k-1}}-1\right)+\text{\ensuremath{\xi}}_{2k-1}\nonumber \\
 & =o(r^{4+4\beta-\epsilon})+R_{2k}+\xi_{2k-1}.\nonumber 
\end{align}
 We have the expansion 
\[
R_{2k}=\sum_{l=0}^{2k}\phi_{l}^{2k}r^{4((l+1)\beta+l)}+O(r^{\frac{4}{2k+1}})
\]
Solve 
\begin{equation}
\Delta^{2}Q_{2k}=\sum_{l=0}^{2k}\phi_{l}^{2k}r^{4((l+1)\beta+l)},\label{eq:Q_2k, connecting eq}
\end{equation}
by Lemma \ref{lem:polynomial lemma}. Let $w_{2k+1}=\bar{w}_{2k}-Q_{2k}$.
By (\ref{eq:Q_2k, connecting eq})
\[
\Delta^{2}w_{2k+1}=o(r^{4+4\beta-\epsilon})+O(r^{\frac{4}{2k+1}})+\xi_{2k-1}.
\]
Hence, the right hand side of (\ref{eq beta =00003D1/2k+1 -1, w_2k})
is $\holder$ continuous. By the regularity theory of elliptic equations,
$w_{2k+1}\in C^{4,\gamma}$ for some $\gamma>0$ and we conclude the
proof. 
\end{proof}
\begin{proof}
[Proof of Case 3,] $\beta=\frac{1}{2k}-1$ . As in previous cases,
we replace $w$ by $\bar{w}=w-g_{0}(x)$ such that $\bar{w}=o(r^{4\beta+4-\epsilon})$. 

If $k=1$, we use Taylor expansion of $e^{4g_{0}(x)}$ at $0$ to
get 
\[
(e^{4g_{0}(x)}-1)=f(x)r^{-2}+O(1),
\]
where $f(x)$ is a polynomial with degree at most 2. Then
\begin{align}
\Delta^{2}\bar{w} & =e^{4\bar{w}+4g_{0}}r^{4\beta}\label{eq:first step beta=00003D-1/2}\\
 & =e^{4g_{0}}(e^{4\bar{w}}-1)r^{4\beta}+(e^{4g_{0}})r^{-2}\nonumber \\
 & =o(r^{8\beta+4-\epsilon})+f(x)r^{4\beta}+\xi(x)\nonumber \\
 & =o(r^{-\epsilon})+f(x)r^{-2}+\xi(x)\nonumber 
\end{align}
where $\xi$ is a smooth function. As in (\ref{eq: finding q_0}),
there exists a function $q_{0}(x)$ by (\ref{eq:form beta=00003D-1/2})
such that 
\begin{equation}
\Delta^{2}(q_{0}r^{2})=f(x)r^{-2}.\label{eq:q_0r^2 connecting}
\end{equation}
By (\ref{eq:first step beta=00003D-1/2}) and (\ref{eq:q_0r^2 connecting}),
\begin{equation}
\Delta^{2}(\bar{w}-q_{0}(x)r^{2})=o(r^{-\epsilon})+\xi(x).\label{eq:Delta^2=00005Cbarw-q_0r^2 =00005Cbeta=00003D-1/2}
\end{equation}
Apply the elliptic regularity theory to (\ref{eq:Delta^2=00005Cbarw-q_0r^2 =00005Cbeta=00003D-1/2}),we
have
\[
w_{1}=\bar{w}-q_{0}(x)r^{2}\in C^{4-\epsilon}.
\]
Suppose that $\bar{w}_{1}=w_{1}-g_{1}(x)$ such that $g_{1}$ is a
polynomial with degree 3 and $\bar{w}_{1}=o(r^{4-\epsilon})$. Then
\begin{align*}
\Delta^{2}\bar{w}_{1} & =\Delta^{2}(w_{1})=\Delta^{2}(\bar{w}-q_{0}r^{2})\\
 & =e^{4w}r^{-2}-f(x)r^{-2}\\
 & =e^{4(g_{1}+q_{0}r^{2}+g_{0})}(e^{4\bar{w}_{1}}-1)r^{-2}+e^{4g_{0}+4g_{1}}(e^{4q_{0}r^{2}}-1)r^{-2}\\
 & +e^{4g_{0}}(e^{4g_{1}}-1)r^{-2}+\xi(x)\\
 & =o(r^{2-\epsilon})+c\log r+R_{1}(x)+\xi(x).
\end{align*}
Here, the remainder term 
\begin{align*}
R_{1}(x) & =e^{4g_{0}}(e^{4g_{1}}-1)r^{-2}\\
 & =f_{1}(x)r^{-2}+O(r),
\end{align*}
where $f_{1}(x)$ is a quadratic polynomial. So 
\[
\Delta^{2}\bar{w}_{1}=\xi_{1}(x)+c\log r+f_{1}(x)r^{-2},
\]
for some $\holder$ continuous $\xi_{1}$. Pick $c_{1}\in\mathbb{R}$
and $q_{1}(x)$ a function in (\ref{eq:form beta=00003D-1/2}) such
that 

\[
\Delta^{2}c_{1}r^{4}\log r=c\log r-c_{2},
\]
and 
\[
\Delta^{2}[q_{1}(x)r^{2}]=f_{1}(x)r^{-2}.
\]
We obtain that
\[
\Delta^{2}(\bar{w}_{1}-c_{1}r^{4}\log r-q_{1}(r)r^{2})=\xi_{2}(x),
\]
where $\xi_{2}$ is a $\holder$ continuous function. This shows that
\[
w=(a_{ij}x_{i}x_{j}+b_{i}x_{i}+c)r^{2}\log r+c_{1}r^{4}\log r+\psi(x),
\]
where $\psi(x)\in C^{4,\gamma}$ for $\gamma>0$. Particularly, it
concludes the case where $\beta=-1/2$. We are done.

For $\beta\not=-1/2$, there exists $q_{0}(x)$ a polynomial satisfying
(\ref{eq: finding q_0}). Then 
\[
w_{1}=\bar{w}-q_{0}(x)r^{4\beta+4}\in C^{8\beta+8-\epsilon}.
\]
There exists a polynomial $g_{1}(x)$ such that $\bar{w}_{1}=w_{1}-g_{1}=o(r^{8\beta+8-\epsilon})$.
We expand $\Delta^{2}\bar{w}_{1}$ in the form of (\ref{eq:=00005CDelta^2 w_j =00005Cbeta=00005Cnpt=00003D-1/2}).
Note that the remaining term $R_{1}(x)$ in (\ref{eq:=00005CDelta^2 w_j remaining terms})
has the following form:
\begin{align}
R_{1}(x) & =\sum_{l=1}^{2k}\phi_{l}r^{4l(\beta+1)-4}+O(r^{\gamma}).\label{eq:R_1 =00005Cbeta=00003D1/2k-1}\\
 & =\sum_{l=1}^{2k-1}\phi_{l}r^{\frac{2l}{k}-4}+\phi_{2k}+O(r^{\gamma}).\nonumber 
\end{align}
In (\ref{eq:R_1 =00005Cbeta=00003D1/2k-1}), $\phi_{k}(x)r^{\frac{2k}{k}-4}=\phi_{k}(x)r^{-2}.$
We assume $\text{deg }\phi_{k}\leq2$, since $p_{m}(x)r^{-2}$ is
$\holder$ continuous for any homogeneous polynomial $p_{m}\in\mathcal{P}_{m}$
for $m\geq3$. We can find polynomials $q_{1}^{l}(x)$ for $l\not=k$,
such that 
\[
\Delta^{2}(q_{1}^{l}(x)r^{4l(\beta+1)})=\phi_{l}(x)r^{4l(\beta+1)-4}.
\]
There exists a function $q_{1}^{k}(x)$ in the form of (\ref{eq:form beta=00003D-1/2})
such that $\Delta^{2}(q_{1}^{k}(x)r^{2})=\phi_{k}r^{-2}$. Let 
\[
Q_{1}(x)=\sum_{l=1}^{2k-1}q_{1}^{l}(x)r^{4l(\beta+1)}.
\]
Let $w_{2}=\bar{w}_{1}-Q_{1}(x)$. Then
\[
\Delta^{2}w_{2}=o(r^{16\beta+12-\epsilon})+\xi_{2}(x).
\]
We iterate as in the proof of Case 1. Suppose that $w_{l}\in C^{4(l+1)(\beta+1)}$,
$l\leq j\leq2k$ and $w_{j}=\bar{w}_{j-1}-Q_{j-1}$, for
\[
Q_{j-1}=\sum_{l=1}^{2k}q_{j-1}^{l}r^{4l(\beta+1)}P_{j-1}^{l}(\log r),
\]
where $q_{j-1}^{l}$ and $P_{j-1}^{l}$ are polynomials. For $w_{j}$
there is a degree 3 polynomial $g_{j}$ such that $\bar{w}_{j}=w_{j}-g_{j}$
and $\bar{w}_{j}=o(r^{4(l+1)(\beta+1)-\epsilon})$. Thus
\begin{align}
\Delta^{2}\bar{w}_{j} & =o(r^{4(j+1)(\beta+1)-\epsilon})+R_{j}(x)+\xi_{j-1}(x).\label{eq:=00005Cbarw_j case 3}
\end{align}
Here $R_{j}(x)$ can be written as
\begin{align}
R_{j}(x) & =\sum_{l\not=k,1\leq l\leq2k}\phi_{j}^{l}(x)\bar{P}_{j}^{l}(\log r)r^{\frac{2l}{k}-4}+\phi_{j}^{k}(x)r^{-2}\bar{P}_{j}^{k}(\log r)+O(r^{\gamma}),\label{eq:R_j case 3,sec 7}
\end{align}
where $\phi_{j}^{l},\bar{P}_{j}^{l}$ are polynomials. We may assume
that $\text{deg }\phi_{j}^{k}\leq2$ because $p_{m}(x)r^{-2}\bar{P}_{j}^{k}(\log r)$
is $\holder$ continuous for any homogeneous polynomial $p_{m}\in\mathcal{P}_{m}$,
$m\geq3$. By Lemma \ref{lem:polynomial lemma}, we can find 
\[
Q_{j}(x)=\sum_{l=1}^{2k}q_{j}^{l}(x)P_{j}^{l}(\log r)r^{4l(\beta+1)},
\]
 such that 
\begin{equation}
\Delta^{2}Q_{j}=\sum_{l\not=k,1\leq l\leq2k}\phi_{j}^{l}(x)\bar{P}_{j}^{l}(\log r)r^{\frac{2l}{k}-4}+\phi_{j}^{k}(x)r^{-2}\bar{P}_{j}^{k}(\log r).\label{eq:Q_j ,case 3; sec 7}
\end{equation}
Then by (\ref{eq:=00005Cbarw_j case 3}),(\ref{eq:R_j case 3,sec 7})
and (\ref{eq:Q_j ,case 3; sec 7}),
\begin{equation}
\Delta^{2}(\bar{w}_{j}-Q_{j})=o(r^{4(j+2)(\beta+1)-4-\epsilon})+\xi_{j}(x).\label{eq:=00005Cbarw_j+1; case 3}
\end{equation}
Let $w_{j+1}=\bar{w}_{j}-Q_{j}$. Then, $w_{j+1}$ is in $W^{4,p}$
for $1<p<-\frac{1}{(j+2)\beta+j+1}$ . The iteration procedure does
not stop until $j=2k-1$ when the right hand side of (\ref{eq:=00005Cbarw_j+1; case 3})
is $\holder$ continuous. Again, we use the elliptic regularity theory
to show $w_{2k}\in C^{4,\gamma}$ and the proof is complete. 
\end{proof}

\section{Uniqueness result with 2 singularities\protect 
}In this section, we give a proof of Theorem \ref{thm:Uniquness result}.
Let $M=S^{4}$. Let $g_{0}$ be the standard metric on 4-sphere. Let
$(M,g_{0},D,g_{1})$ be the conic sphere with divisor $D=\beta_{0}p_{0}+\beta_{1}p_{1}$,
then $\int_{M}Q_{g_{1}}dV_{g_{1}}=8\pi^{2}(2+\beta_{0}+\beta_{1})$
by Proposition \ref{prop:(Gauss-Bonnet-Chern)-Suppose-tha}. Note
if $w$ is a solution on the sphere with divisor $D$, we have
\[
P_{S^{4}}w+6=3(2+\beta_{0}+\beta_{1})e^{4w}.
\]
Here we have normalized the equation such that the conic sphere has
the same volume as that of a standard 4-sphere. Let $\bar{k}_{g}=3(2+\beta_{0}+\beta_{1})$.
We only consider solutions such that 
\[
w-\sum_{i=0,1}\eta_{i}(x)\beta_{i}\log|x-p_{i}|\in H^{2}(dV_{0}),
\]
where $p_{1}$ and $p_{2}$ are two points on the sphere and $\eta_{i}(x)$
are cut off functions in the neighborhood of $p_{i}$ in (\ref{eq:singular term}),
$i=0,1$ respectively. By a conformal transform on the sphere, we
may assume that $p_{1}$ and $p_{2}$ are antipodes $x_{S}$ and $x_{N}$,
respectively. By a stereographic projection from $x_{N}$, we obtain
the equation on $\mathbb{R}^{4}$ 
\begin{equation}
\Delta^{2}u=\bar{k}_{g}e^{4u}.\label{eq:equation in sec 8,}
\end{equation}
We state two lemmas that describe the asymptotic behavior of $u$. 
\begin{lem}
$\Delta u-2\beta_{0}\frac{1}{|x|^{2}}=-\frac{\bar{k}_{g}}{4\pi^{2}}\int_{\mathbb{R}^{4}}\frac{e^{4u(y)}}{|x-y|^{2}}dy-C_{1}$
where $C_{1}\geq0$ is a constant. \label{lem:Lemma 2.2}
\end{lem}

\begin{lem}
\label{lem:Lemma 2.5}$u-\beta_{0}\log|x|=-\frac{\bar{k}_{g}}{8\pi^{2}}\int_{\mathbb{R}^{4}}\log\frac{|x-y|}{|y|}e^{4u(y)}dy+C_{0}$
where $C_{0}$ is a constant. Besides, for any $\epsilon>0$ there
is an $R_{\epsilon}$ such that 
\begin{equation}
(-2-\beta_{1})\log|x|\leq u\leq(-2-\beta_{1}+\epsilon)\log|x|,\label{eq:Lin' second lemma}
\end{equation}
for $|x|\geq R_{\epsilon}$.
\end{lem}

For the proof of Lemma \ref{lem:Lemma 2.2} and Lemma \ref{lem:Lemma 2.5},
see Lemma 2.1 - 2.5 in \cite{lin1998classification}. We should mention
that since we always assume that the solution $u$ comes from a $H^{2}$
function on $S^{4}$, $u$ satisfies assumptions in Lin's paper for
both lemmas. 

Now we derive an asymptotic expansion of $u$ at infinity.

\begin{lem}
\label{lem:Lemma 2.8 Asymptotic expansion at infinity}Let $u$ be
a solution of (\ref{eq:equation in sec 8,}). Then, $u$ has the following
asymptotic expansion as $\infty$:
\[
u(x)=-(2+\beta_{1})\log|x|+c+O(|x|^{-1})
\]
 and  
\begin{align}
\begin{cases}
-\Delta u(x)=|x|^{-2}(a_{0}+\sum_{4l(\beta_{1}+1)<1}a_{0,l}|x|^{-4l(\beta_{1}+1)}P_{0,l}(-\log|x|)+\sum_{i}a_{i}x_{i}|x|^{-2}\\
+\sum_{4l(\beta_{1}+1)<1}\sum_{i=1}^{4}a_{i,l}x_{i}|x|^{-4l(\beta_{1}+1)-2}P_{i,l}(-\log|x|))+O(|x|^{-4}),\\
-\frac{\partial}{\partial x_{i}}\Delta u(x)=a_{0}x_{i}|x|^{-4}+O(|x|^{-(3+\delta)}),\\
-\frac{\partial^{2}}{\partial x_{i}\partial x_{j}}\Delta u(x)=O(|x|^{-4}),
\end{cases}\label{eq:Asymptotic expansion at infinity}
\end{align}
for large $|x|$, where $c,\ 0<\delta<4(\beta_{1}+1)$ and $a_{i,l}$
are constants and $P_{i,l}$ are polynomials. Note that $a_{0}=2(2+\beta_{1})$
is positive.
\end{lem}

\begin{proof}
Let $w(x)=u\left(\frac{x}{|x|^{2}}\right)-(2+\beta_{1})\log|x|$.
By Lemma \ref{lem:Lemma 2.2} and Lemma \ref{lem:Lemma 2.5}, we see
that $w(x)$ satisfies
\[
\begin{cases}
\Delta^{2}w(x)=\bar{k}_{g}e^{4w}|x|^{4\beta_{1}} & in\ \mathbb{R}^{4}-\{0\},\\
|w(x)|=o(\log|x|) & as\ |x|\to0,\\
|\Delta w|=o(|x|^{-2}) & as\ |x|\to0.
\end{cases}
\]
Let $h(x)$ be a weak solution of 
\[
\begin{cases}
\Delta^{2}h(x)=\bar{k}_{g}e^{4w(x)}|x|^{4\beta_{1}} & in\ B_{1},\\
h(x)=w(x) & on\ \partial B_{1},\\
\Delta h(x)=\Delta w(x) & on\ \partial B_{1}.
\end{cases}
\]
By Lemma \ref{lem:Lemma 2.5}, $e^{4w(x)}|x|^{4\beta_{1}}$ is in
$L^{p}(B_{1})$ for $(-\beta_{1})^{-1}>p>1$. By the regularity theory
of elliptic equations, $\Delta h(x)\in W^{2,p}(B_{1})$ and $h(x)\in W^{4,p}(B_{1})$
and hence $h(x)\in C^{4\tau}$ for $0<\tau<(1+\beta_{1})$. Now, let
$q(x)=w(x)-h(x)$. Then it satisfies that 
\[
\begin{cases}
\Delta^{2}q=0 & in\ B_{1}-\{0\},\\
q=\Delta q=0 & on\ \partial B_{1},\\
|q(x)|=o(\log|x|),|\Delta q|=o(|x|^{-2}) & as\ |x|\to0.
\end{cases}
\]
Thanks to the asymptotic property, we can still apply maximum principle
to $\Delta q$ which implies $\Delta q\equiv0$ and similarly, $q\equiv0$.
Therefore, $w(x)=h(x)$. 

Applying the asymptotic expansion in Theorem \ref{thm:Asymptotic expansion at singular points},
the lemma follows immediately.
\end{proof}
We apply the moving plane method to prove the radial symmetry of solutions
with two conical singularities. Following the convention in the literature,
see for example \cite{gidas1979,caffarelli1989asymptotic,lin1998classification},
let $\lambda\in\mathbb{R}$, $T_{\lambda}=\{(x_{1},x_{2},x_{3},x_{4}):x_{1}=\lambda\}$,
$\Sigma_{\lambda}=\{x:x_{1}>\lambda\}$, and $x^{\lambda}=(2\lambda-x_{1},x_{2},x_{3},x_{4})$.
In order to initiate the moving plane in $x_{1}$ direction, we need
the following two lemmas.
\begin{lem}
Let $v$ be a positive function defined in a neighborhood of infinity
satisfying the asymptotic expansion (\ref{eq:Asymptotic expansion at infinity}).
Then there exists $\bar{\lambda}$ and $R>0$ such that 
\[
v(x)>v(x^{\lambda})
\]
holds for $\lambda<\bar{\lambda}$, $|x|\geq R$ and $x\in\Sigma_{\lambda}$.
\label{lem:Monontonicity lemma}
\end{lem}

\begin{lem}
Suppose v satisfies the assumption of Lemma \ref{lem:Monontonicity lemma}
and $v(x)>v(x^{\lambda_{0}})$ for $x\in\Sigma_{\lambda_{0}}$. Assume
$v(x)-v(x^{\lambda_{0}})$ is superharmonic in $\Sigma_{\lambda_{0}}$.
Then there exist $\epsilon>0,S>0$ such that the followings hold.

(i) $v_{x_{1}}>0$ in $|x_{1}-\lambda_{0}|<\epsilon$ and $|x|>S$.

(ii) $v(x)>v(x^{\lambda})$ in $x_{1}\geq\lambda_{0}+\frac{\epsilon}{2}>\lambda$
and $|x|>S$\\
for all $x\in\Sigma_{\lambda},\lambda\leq\lambda_{1}$ with $|\lambda_{1}-\lambda_{0}|<c\epsilon$,
where $c=c(\lambda_{0},v)$ is a small positive number.\label{lem: Abnormal points are Bounded sequence}
\end{lem}

Both lemmas are contained in the celebrated paper by Caffarelli-Gidas-Spruck
\cite{caffarelli1989asymptotic}. For the proofs, see Lemma 2.3 and
2.4 in \cite{caffarelli1989asymptotic}. We should remind the readers
that although our asymptotic expansion is not the exact form in the
above paper, the leading terms are the same. Hence the argument in
\cite{caffarelli1989asymptotic} can be applied. 
\begin{proof}
[Proof of Theorem \ref{thm:Uniquness result}]For any $\lambda\not=0$,
let $w_{\lambda}(x)=u(x)-u(x^{\lambda})$ in $\Sigma_{\lambda}$.
Then $w_{\lambda}(x)$ satisfies
\[
\begin{cases}
\Delta^{2}w_{\lambda}=b_{\lambda}(x)w_{\lambda} & x\in\Sigma_{\lambda},\\
w_{\lambda}=\Delta w_{\lambda}=0 & x\in T_{\lambda},
\end{cases}
\]
where 
\[
b_{\lambda}(x)=\bar{k}_{g}\frac{e^{4u(x)}-e^{4u(x^{\lambda})}}{u(x)-u(x^{\lambda})}>0.
\]
By our assumption at $0$, 
\begin{equation}
w_{\lambda}(x)=\beta_{0}\log|x|+O(1).\label{eq:w_=00005Clambdabigatx^*_=00005Clambda}
\end{equation}
By Lemma \ref{lem:Monontonicity lemma}, $-\Delta w_{\lambda}>0$
for $x\in\Sigma_{\lambda}$, $\lambda\leq\bar{\lambda}<0,\ |x|>R$.
Since $v(x)=-\Delta u>0$, there is $\bar{\lambda}_{1}\leq\bar{\lambda}$
such that $v(x^{\lambda})<v(x)$ for $|x|<R$ and $\lambda<\bar{\lambda}_{1}$.
Hence 
\[
-\Delta w_{\lambda}(x)>0,
\]
in $\Sigma_{\lambda}$ for $\lambda\leq\bar{\lambda}_{1}.$ We remark
that the singularity of $u$ at $0$ does not affect the computation.
By Lemma \ref{lem:Lemma 2.8 Asymptotic expansion at infinity}, $\lim_{|x|\to\infty}w_{\lambda}(x)=0$.
By (\ref{eq:w_=00005Clambdabigatx^*_=00005Clambda}) we can choose
$\delta$ small such that $w_{\lambda}>0$ on the boundary $\partial B_{\delta}(0)$.
Apply maximum principle in $\Sigma_{\lambda}-B_{\delta}(0)$ , we
have $w_{\lambda}(x)>0$ in $\Sigma_{\lambda}-B_{\delta}(0)$ for
$\lambda\leq\bar{\lambda}_{1}$. Then taking $\delta\to0$, we have
\begin{equation}
w_{\lambda}(x)>0,x\in\Sigma_{\lambda}.\label{eq:w_=00005Clambda=00300B0}
\end{equation}
Let 
\[
\lambda_{0}=\sup\{\lambda<0:v(x^{\mu})\leq v(x),\ x\in\Sigma_{\mu}\ for\ \mu\leq\lambda\}.
\]
If $\lambda_{0}=0$ then we are done. Otherwise, we claim that 
\begin{equation}
u(x)\equiv u(x^{\lambda_{0}})\label{eq:symmetry of u in sec 8}
\end{equation}
for $x\in\Sigma_{\lambda_{0}}$. This also implies that $\lambda_{0}=0$.
We argue by contradiction. Suppose that $\lambda_{0}<0$ and $w_{\lambda_{0}}\not\equiv0$
in $\Sigma_{\lambda_{0}}$. By continuity, $\Delta w_{\lambda_{0}}\leq0(=-\infty\ at\ 0)$
in $\Sigma_{\lambda_{0}}$. Since $w_{\lambda_{0}}(x)\to0$ as $|x|\to\infty$,
by strong maximum principle $w_{\lambda_{0}}>0$ in $\Sigma_{\lambda_{0}}$.
Then we have 
\begin{equation}
\Delta^{2}w_{\lambda_{0}}=\bar{k}_{g}(e^{4u}(x)-e^{4u(x^{\lambda_{0}})})>0.\label{eq:w_=00005Clambda_0 subharmonic}
\end{equation}
Hence $\Delta w_{\lambda_{0}}$ is subharmonic. Similar to (\ref{eq:w_=00005Clambda=00300B0}),
we apply strong maximum principle in $\Sigma_{\lambda_{0}}$ to get
$\Delta w_{\lambda_{0}}(x)<0$ in $\Sigma_{\lambda_{0}}$. By the
definition of $\lambda_{0}$ , there is a sequence $\lambda_{n}\downarrow\lambda_{0}$
and $\lambda_{n}<0$ such that $\sup_{\Sigma_{\lambda_{n}}}\Delta w_{\lambda_{n}}>0$.
Since $\lim\limits _{|x|\to\infty}\Delta w_{\lambda_{n}}(x)=0$, there
exists $z_{n}\in\Sigma_{\lambda_{n}}$ such that 
\[
\Delta w_{\lambda_{n}}(z_{n})=\sup_{x\in\Sigma_{\lambda_{n}}}\Delta w_{\lambda_{n}}(x)>0.
\]
Note that clearly $z_{n}\not=0$ and at each $z_{n}$, 
\[
\nabla\Delta w_{\lambda_{n}}(z_{n})=0.
\]
By Lemma \ref{lem: Abnormal points are Bounded sequence}, $z_{n}$
are bounded. Suppose that $z_{0}$ is a limit point of $z_{n}$. $z_{0}$
can not be $0$ since $\Delta w_{\lambda}(z)\to-\infty$ as $z\to0$.
If $z_{0}\in\Sigma_{\lambda_{0}}$, by continuity, $\Delta w_{\lambda_{0}}(z_{0})=0$.
This contradicts with the fact that $\Delta w_{\lambda_{0}}<0$ in
$\Sigma_{\lambda_{0}}$. If $z_{0}\in T_{\lambda_{0}}$, then $\nabla(\Delta w_{\lambda_{0}}(z_{0}))=0$,
which along with (\ref{eq:w_=00005Clambda_0 subharmonic}) contradicts
to Hopf's lemma at $z_{0}$. Hence the claim is proved.

By (\ref{eq:symmetry of u in sec 8}), we have 
\[
u(x_{1},x_{2},x_{3},x_{4})=u(-x_{1},x_{2},x_{3},x_{4}).
\]
By choosing different coordinate systems, we get that $u$ is radial
symmetric. We have thus proved Theorem \ref{thm:Uniquness result}.
\end{proof}

\section*{Appendix}
\begin{proof}
[Proof of Lemma \ref{lem:polynomial lemma}] For each degree $m$,
we only have to consider homogeneous polynomials in $\mathcal{P}_{m}$.We
discuss 2 cases where $\beta\neq-\frac{1}{2}$ or $\beta=-\frac{1}{2}$.

Case 1. Suppose that $\beta\not=-\frac{1}{2}$. Let $p(x)\in\mathcal{P}_{m}$
. By Euler formula, $x_{i}D_{i}p(x)=mp(x)$. 
\begin{align*}
\Delta(p(x)r^{4\beta+2}) & =r^{4\beta}\left(r^{2}\Delta p(x)+(4\beta+2)\left(4\beta+4+2m\right)p(x)\right)
\end{align*}
So $\Delta\left(p(x)r^{4\beta+2}\right)=f(x)r^{4\beta}$ if 
\begin{equation}
r^{2}\Delta p(x)+(4\beta+2)(4\beta+4+2m)p(x)=f(x).\label{eq:beta =00005Cnot=00003D-1/2 connecting 1, app}
\end{equation}
Note if $\beta\in(-1,0)$ and $\beta\not=-1/2$, then $(4\beta+2)(4\beta+4+2m)$
can not be an eigenvalue of $r^{2}\Delta$. Then by Lemma \ref{lem:(e.g.-)-Lee-Park},
we see that there exists a $p(x)$ such that 
\begin{equation}
\Delta(p(x)r^{4\beta+2})=f(x)r^{4\beta}.\label{eq:beta =00005Cnot=00003D-1/2 connceting 2;app}
\end{equation}
Now let $q(x)$ be a homogeneous polynomial with degree $m$. Then
\[
\Delta(q(x)r^{4\beta+4})=r^{4\beta+2}\left(r^{2}\Delta q(x)+(4\beta+4)\left(4\beta+6+2m\right)q(x)\right)
\]
Apply Lemma \ref{lem:(e.g.-)-Lee-Park} again, there is a polynomial
such that $\Delta\left(q(x)r^{4\beta+4}\right)=p(x)r^{4\beta+2}$.

To solve 
\begin{equation}
\Delta^{2}g=f(x)(\log r)^{k}|x|^{4\beta},\label{eq:Delta ^2g=00003Df, connecting eq;app}
\end{equation}
 we first compute
\begin{align}
\Delta[p(x)r^{4\beta+2}\log r] & =r^{4\beta}\log r\left(r^{2}\Delta p(x)+(4\beta+2)(2m+4\beta+4)p(x)\right)\label{eq:beta=00005Cnot=00003D-1/2;=00005Clogr}\\
 & +r^{4\beta}\phi(x),\nonumber 
\end{align}
where $\phi(x)\in\mathcal{P}_{m-2}$. First, we can solve 
\[
r^{2}\Delta p_{1}(x)+(4\beta+2)(2m+4\beta+4)p_{1}(x)=f(x),
\]
as in (\ref{eq:beta =00005Cnot=00003D-1/2 connecting 1, app}). We
can also find a polynomial $p_{2}(x)$ such that $\Delta(p_{2}(x)r^{4\beta+2})=\phi(x)r^{4\beta}$.
Thus, 
\[
\Delta[r^{4\beta+2}(p_{1}\log r-p_{2})]=f(x)r^{4\beta}\log r.
\]
By a similar argument, there exist $q_{1},q_{2}$ such that 
\[
\Delta[r^{4\beta+4}(q_{1}(x)\log r+q_{2}(x))]=(p_{1}(x)|x|^{2}\log r-p_{2}(x)|x|^{2})r^{4\beta},
\]
hence
\[
\Delta^{2}[(q_{1}(x)\log r+q_{2}(x))r^{4\beta+4}]=f(x)(\log r)r^{4\beta}.
\]
This gives the solution $g=r^{4\beta+4}(q_{1}(x)\log r+q_{2}(x))$
for $k=1$ in (\ref{eq:Delta ^2g=00003Df, connecting eq;app}). We
use induction for $k\geq2$. Suppose that we can find solutions of
(\ref{eq:log^l}) for $0\leq l\leq k-1$
\begin{align}
\Delta\left(p(x)r^{4\beta+2}(\log r)^{k}\right) & =r^{4\beta}(\log r)^{k}\left(r^{2}\Delta p(x)+(4\beta+2)(4\beta+4+2m)p(x)\right)\label{eq:=00005CDeltap(x)r^=00007B4=00005Cbeta+2=00007D(=00005Clogr)^=00007Bk=00007D=00005Crightconnectingeq1; app}\\
 & +\sum_{j=0}^{k-1}r^{4\beta}\phi_{j}(x)(\log r)^{j},\nonumber 
\end{align}
where $\phi_{j}(x)$ are polynomials. Then we can find $p(x)$ such
that 
\[
r^{2}\Delta p(x)+(4\beta+2)(4\beta+4+2m)p(x)=f(x).
\]
In the remaining terms $\sum_{j=0}^{k-1}r^{4\beta}\phi_{j}(x)(\log r)^{j}$,
the degrees of $\log r$ are strictly smaller than $k$. Therefore,
the remaining terms can be solved by induction and there exist $\{p_{l}(x)\}$
such that 
\[
\Delta\sum_{l}p_{l}r^{4\beta+2}(\log r)^{l}=f(x)r^{4\beta}(\log r)^{k}.
\]
Repeat the argument for each $\tilde{f}(x)=p_{l}(x)|x|^{2}$ and we
can solve (\ref{eq:Delta ^2g=00003Df, connecting eq;app}).

Case 2. $\beta=-\frac{1}{2}$. If homogeneous degree $m=0$, we see
that $\Delta^{2}(\frac{c}{16}r^{2}\log r)=cr^{-2}$. So this is true
for degree $0$ polynomial. If $m=1$, direct computation shows that
\[
\Delta^{2}(\frac{a_{i}x_{i}}{48}r^{2}\log r)=a_{i}x_{i}.
\]
If $m=2$, we have for $i\not=j$,
\[
\Delta^{2}(x_{i}x_{j}r^{2}\log r)=96x_{i}x_{j}r^{-2}.
\]
For $i=j$, we compute
\[
\Delta^{2}(x_{i}^{2}r^{2}\log r)=32+96x_{i}^{2}r^{-2}+48\log r.
\]
Note that $\Delta^{2}(r^{4}\log r)=7\times64+3\times128\log r$. Since
$\Delta^{2}r^{4}=192$, we can still find a solution for $\Delta^{2}(q(x)r^{2})=x_{i}^{2}r^{2}$
in the form of (\ref{eq:form beta=00003D-1/2}).

For functions in the form of $f(x)r^{-2}(\log r)^{l}$ in (\ref{eq:beta=00003D-1/2 and (logr)^k}),
we argue by induction with respect to $k$. Note the above argument
is for $l=0$. Suppose that for $0\leq l\leq k-2,$ we have a solution
for (\ref{eq:beta=00003D-1/2 and (logr)^k}). We prove for $l=k-1$.
Any quadratic polynomial $f(x)$ is a linear combination of $1,x_{i},x_{i}x_{j},x_{i}^{2}$.
Thus, we only have to consider these 4 subcases. 

Subcase 1, $f(x)=1$. Take test function $cr^{2}(\log r)^{k}$. Compute
\begin{align}
\Delta^{2}[cr^{2}(\log r)^{k}] & =cr^{-2}(c_{1}(\log r)^{k-4}+c_{2}(\log r)^{k-3}+c_{3}(\log r)^{k-2}+2^{4}(\log r)^{k-1}),\label{eq:beta =00003D-1/2 aonnecting eq; app}
\end{align}
where $c_{i}$ are polynomials of $k$ and $c_{i}=0$ for $k=1,2,...,4-i$.
Let $c=2^{-4}$ and the first term can be cancelled. The remaining
terms $f-\Delta^{2}[cr^{2}(\log r)^{k}]$ in (\ref{eq:beta =00003D-1/2 aonnecting eq; app})
are lower degrees terms and can be solved by induction. 

Subcase 2, $f(x)=x_{i}$. We consider a test function of the form
$cx_{i}r^{2}(\log r)^{k}$. Direct computation shows
\begin{align*}
\Delta^{2}(cx_{i}r^{2}(\log r)^{k}) & =ck\frac{x_{i}}{r^{2}}\log^{k-4}\left(r\right)(\sum_{i=0}^{2}c_{i}(\log r)^{i}+48\log^{3}\left(r\right)).
\end{align*}
Likewise, take $c=(48k)^{-1}$. Then, $f(x)-\Delta^{2}(cx_{i}r^{2}(\log r)^{k})$
has lower degrees in $\log r$ and we can solve the remaining terms
by induction. 

Subcase 3, $f(x)=x_{i}x_{j},\ i\not=j$. By taking out a test function
in the form of $cx_{i}x_{j}r^{2}(\log r)^{k}$ with a proper choice
of $c$, we may reduce the top degree of $\log r$. Then, the lower
degree terms can be solved by induction. The argument is similar to
subcase 2.

Subcase 4, $f(x)=x_{i}^{2}$. We compute directly,
\begin{align}
\Delta^{2}(x_{i}^{2}r^{2}(\log r)^{k}) & =x_{i}^{2}r^{-2}(\log r)^{k-1}k\times96+x_{i}^{2}r^{-2}kP(\log r)+Q(\log r),\label{eq:=00005CDeltax_i^2r^2(=00005Clogr)^k, app}
\end{align}
where $P$ and $Q$ are polynomials in one variable and $\deg P(x)\leq k-2,\deg Q(x)\leq k$.
Note that the degree of $P(x)$ is less than $k-1$. Thus, by induction
assumption there exists a function $\bar{P}(x)$ in the form of (\ref{eq:beta=00003D-1/2 and (logr)^k})
such that 
\begin{equation}
\Delta^{2}(\bar{P}(x)r^{2})=x_{i}^{2}r^{-2}P(\log r).\label{eq:Delta^2 P}
\end{equation}
For $Q(x)$ in (\ref{eq:=00005CDeltax_i^2r^2(=00005Clogr)^k, app}),
if it has degree $k$, there is a test function in the form of $cr^{4}(\log r)^{k}$.
We compute
\begin{align*}
\Delta^{2}\left(r^{4}(\log r)^{k}\right) & =k\log^{k-4}\left(r\right)(\left(k-1\right)(k-2)(k-3)\\
 & +b_{1}k\left(k-1\right)(k-2)\log\left(r\right)\\
 & +b_{2}(k-1)k\log^{2}\left(r\right)+b_{3}k\log^{3}\left(r\right))+12\times16\log^{k}\left(r\right),
\end{align*}
where $b_{i}$ are constants independent of $k$. Suppose that $a_{k}\not=0$
is the coefficient of the top degree term of $Q(x)$. Then, 
\begin{equation}
R(\log r)=Q(\log r)-\Delta^{2}\left(\frac{a_{k}}{12\times16}r^{4}(\log r)^{k}\right)\label{eq:R(=00005Clogr)}
\end{equation}
is a polynomial in $\log r$ with degree less than $k$. Thus, we
can apply the induction assumption to $R(\log r)$ to show that there
exists a polynomial $\bar{Q}(x)$ with $\deg\bar{Q}(x)\leq k$ and
\begin{equation}
\Delta^{2}(r^{4}\bar{Q}(\log r))=Q(\log r).\label{eq:=00005CDelta=00005CbarQ}
\end{equation}
Now let 
\[
\tilde{f}(x)=\frac{1}{96k}(x_{i}^{2}r^{2}(\log r)^{k}-\bar{P}(x)r^{2}-r^{4}\bar{Q}(\log r)).
\]
By (\ref{eq:=00005CDeltax_i^2r^2(=00005Clogr)^k, app}),(\ref{eq:Delta^2 P})
and (\ref{eq:=00005CDelta=00005CbarQ}), 
\[
\Delta^{2}\tilde{f}(x)=x_{i}^{2}r^{-2}(\log r)^{k-1}
\]
This completes the last case.
\end{proof}
\bibliographystyle{amsalpha}
\bibliography{PDE-Fang-Ma}

\end{document}